\documentclass{amsart}
\usepackage{amsmath, amsfonts, latexsym, array, amssymb, amscd }
\usepackage[all]{xypic}
\usepackage[all]{xy}
\usepackage{graphicx}
\usepackage{color}
\usepackage{enumitem}
\usepackage{epstopdf}
\usepackage{hyperref}
\usepackage{mathtools}
\usepackage{mathrsfs}
\usepackage{parskip}
\usepackage{accents}
\usepackage[svgnames]{xcolor}
\usepackage{bbold}

\usepackage{ulem}

\numberwithin{equation}{section}
\numberwithin{figure}{section}

\newtheorem{theorem}{Theorem}[section]
\newtheorem*{theorem*}{Theorem}
\newtheorem{thma}{Theorem}

\newtheorem{lemma}[theorem]{Lemma}

\newtheorem*{proposition*}{Proposition}

\theoremstyle{definition}
\newtheorem{definition}[theorem]{Definition}
\newtheorem{hypothesis}[theorem]{Hypothesis}

\newtheorem{remark}[theorem]{Remark}

\usepackage{indentfirst}
\usepackage{ulem}

\DeclareMathOperator{\Reg}{Reg}
\newcommand{\Sing}{\mathrm{Sing}}

\newcommand{\CCC}{\mathcal{C}}
\newcommand{\BBB}{\mathcal{B}}

\newcommand{\eps}{\epsilon}

\newcommand{\NN}{\mathbb{N}}
\newcommand{\RR}{\mathbb{R}}
\newcommand{\mukbm}{\mu_{\mathrm{KBM}}}

\newcommand{\Per}{\mathrm{Per}}

\newcommand{\ul}{\underline}

\title[Fluctuations of time averages in non-positive curvature]{Fluctuations of time averages around
 closed geodesics in non-positive curvature}
\author{Daniel~J.~ Thompson and Tianyu Wang }
\date{\today}
\subjclass[2010]{37DA50, 37D40, 37D25}
\address{D.~J.~Thompson, Department of Mathematics, The Ohio State University, Columbus, OH 43210, \emph{E-mail address:} \tt{thompson@math.osu.edu}}
\address{T.~Wang, Department of Mathematics, The Ohio State University, Columbus, OH 43210, \emph{E-mail address:} \tt{wang.7828@buckeyemail.osu.edu}}
\begin{document}
\thanks{This work is partially supported by NSF grants DMS-$1461163$ and DMS-$1954463$.}
\maketitle

\begin{abstract}
We consider the geodesic flow for a rank one non-positive curvature closed manifold.  We prove an asymptotic version of the Central Limit Theorem for families of measures constructed from regular closed geodesics converging to the Bowen-Margulis-Knieper measure of maximal entropy.  The technique expands on ideas of Denker, Senti and Zhang, who proved this type of  asymptotic Lindeberg Central Limit Theorem on periodic orbits for expansive maps with the specification property. We extend these techniques from the uniform to the non-uniform setting, and from discrete-time to continuous-time. We consider H\"older observables subject only to the Lindeberg condition and a weak positive variance condition. If we assume a natural strengthened positive variance condition, the Lindeberg condition is always satisfied. Our results extend to dynamical arrays of H\"older observables, and to weighted periodic orbit measures which converge to a unique equilibrium state. \end{abstract}

\section{Introduction}

A goal in the study of dynamical systems with some hyperbolicity is to exhibit the kind of stochastic behavior obeyed by sequences of i.i.d. random variables. In settings with non-uniform hyperbolicity, we may be able to demonstrate this kind of stochastic behavior \emph{within} the system even in situations where it is intractable to demonstrate globally. Our paper follows this philosophy. We consider the geodesic flow for a rank one non-positive curvature closed manifold. We exhibit sequences of measures constructed from regular closed geodesics whose first order behavior is that of the measure of maximal entropy, and whose second order behavior obeys, in the limit, the Lindeberg Central Limit Theorem.

The Lindeberg condition is a classical criteria from Probability Theory, which often gives a necessary and sufficient criteria for the Central Limit Theorem (CLT) to hold for sequences of independent random variables which are not identically distributed. Roughly, the Lindeberg condition guarantees that the variance of a single random variable is negligible in comparison to the sum of all the variances. This idea was recently explored by Denker, Senti and Zhang \cite{DSZ18} in the setting of maps with the specification property. They showed that a Lindeberg condition on the sequence of periodic orbit measures is equivalent to a Central Limit Theorem in the limit. 

The analysis of this paper extends the ideas of Denker, Senti and Zhang to the geodesic flow on a rank one non-positive curvature closed manifold. This is one of the main classes of examples of non-uniformly hyperbolic flows. While the theory of equilibrium states in this setting has been extended recently by \cite{BCFT}, the statistical properties of these measures remain largely wide open, even for the Knieper-Bowen-Margulis measure of maximal entropy $\mukbm$. This contrasts with the well-understood case of geodesic flow on negative curvature manifolds, for which the CLT was established by Ratner \cite{mR73}. In particular, the CLT for the MME and other equilibrium states remains out of reach of current methods in the non-positive curvature setting. %We remark that low regularity of the unstable and stable leaves seems to be a major obstacle towards employing the powerful smooth techniques which are often used for this kind of analysis. 

In this paper, we show that for a H\"older observable, the time averages for certain measures constructed from regular closed geodesics asymptotically obey the Central Limit Theorem. This enriches the picture for these time averages, whose first order behavior is convergence to the integral with respect to the measure of maximal entropy. This result applies under the Lindeberg condition and a weak positive variance condition on the sequence of periodic orbit measures. This result is stated formally as Theorem \ref{main}.  We show that the Lindeberg condition is always satisfied under a natural strengthening of the positive variance condition. This is carried out in \S \ref{Verifylind}. We now build up some notation to state and motivate our results, and give an idea of the constructions involved.

Recall that for an invariant measure $\mu$, and an observable $f$, the dynamical variance for the flow $(g_t)$, when it exists, is defined by
\begin{equation} \label{dynvar}
\sigma^2_{\text{Dyn}}(f, \mu) = \lim_{T \to \infty} \int \left (\frac{F(\cdot, T) - \int F( \cdot, T )d \mu}{\sqrt{T}} \right)^2 d \mu,
\end{equation}
where $F(x, T) = \int_0^T f(g_sx) ds$. In our setting, for a fixed $\eta>0$, we construct a sequence of discrete probability measures $(m_l)$ on closed orbits in $T^1M$ corresponding to uniformly $\eta$-regular closed geodesics (which are defined in \S \ref{countingclosed}). We consider the collection of $\eta$-regular closed geodesics which have least period in the interval $(T_l- \delta_l, T_l]$, where $T_l \to \infty$ and $\delta_l \to 0$, which we denote $\Per^{\eta}_R(T_l - \delta_l, T_l]$. We define $m_l$ by choosing one point in $T^1M$ tangent to each such geodesic (we denote this set of points by $E_l$), and distributing mass equally over these points. By analogy with \eqref{dynvar}, it is natural for us to define the (lower)  dynamical variance for the sequence of measures $(m_l)$ to be
\[
 \underline \sigma^2_{\text{Dyn}}(f, (m_l)) = \liminf_{l \to \infty} \int \left (\frac{F(\cdot, T_l) - \int F( \cdot, T_l )d m_l}{\sqrt{T_l}} \right)^2 d m_l.
\]
We choose two more sequences $k_l \to \infty$, $C_l \to \infty$, and define another sequence of measures $(\nu_l)$. Each $\nu_l$ is given by constructing points out of the product $E_l^{k_l}$ by using a  certain specification property on the $\eta$-regular closed geodesics to find an orbit segment which loops $C_l$ times round each of the closed geodesics indexed by an element of  $E_l^{k_l}$. We write $S_l$ for the total length of an orbit segment specified in this way (precisely, $S_l= k_l(C_l T_l + M)$, where $M$ is the transition time in applying our specification property). The measure $\nu_l$ is given by putting mass equally along the initial segment of length $T_l$ of all the orbit segments defined this way. 

If the variance quantity $\underline \sigma^2_{\text{Dyn}}(f, (m_l)) $ is positive, we can choose $k_l$ and $C_l$ so that the family of measures $(\nu_l)$ satisfies an asymptotic central limit theorem for the observable $f$. We can state a simple version of our main results as follows.

\begin{thma} \label{intromain}
For any $\eta>0$ and sequences $\delta_l \to 0$, $T_l \to \infty$, we define a sequence of discrete probability measures $(m_l)_{l \in \NN}$ by choosing a point tangent to each element of $\Per^{\eta}_R(T_l - \delta_l, T_l]$, and assigning each of these points equal mass. We assume that $T_l$ is chosen to increase sufficiently fast, depending on $\eta$ and $\delta_l$, to allow for our construction of $(\nu_l)$ (see Hypothesis \ref{condstar}). Suppose $f\in C(T^1M)$ is H\"older continuous with
\begin{equation} \label{posvar0}
\begin{aligned}
 \underline \sigma^2_{\mathrm{Dyn}}(f, (m_l)) > 0.
\end{aligned}
\end{equation}
Then there exists sequences $k_l \to \infty$, $C_l \to \infty$, so that the sequence of measures $(\nu_l)$ defined by the data $(\delta_l, T_l, k_l, C_l)_{l\in \NN}$ (see \S \ref{construction} for details of the construction), which converges weak$^\ast$ to the measure of maximal entropy $\mukbm$, satisfies the following asymptotic central limit theorem. For all $a\in \mathbb{R}$, 
\begin{equation} \label{LCLT0}
\lim_{l\rightarrow \infty}\nu_l \left ( \left \{v: \frac{F(v, S_l) - S_l \int f d \nu_l}{\sigma_{\nu_l}(F( \cdot, S_l))}\leq a \right \} \right )=N(a),
\end{equation}
 where $N$ is the cumulative distribution function of the normal distribution $\mathcal{N}(0,1)$, and $\sigma^2_\mu(\phi)$ denotes the usual `static' variance $\sigma^2_\mu(\phi) = \int \left (\phi - \int \phi d\mu \right )^2 d\mu$.
\end{thma}
 The sequences $(k_l)$ and $(C_l)$ are determined by $\underline \sigma^2_{\text{Dyn}}(f, (m_l))$. Thus, given $\alpha>0$, we can find a sequence of measures $(\nu_l)$, defined by the data $(\delta_l, T_l, k_l, C_l)_{l\in \NN}$, so that any H\"older continuous observable $f$ with $\underline \sigma^2_{\text{Dyn}}(f, (m_l)) > \alpha$  satisfies \eqref{LCLT0}. We comment on this positive variance condition. If the manifold has strictly negative curvature, $(m_l)$ places mass on each closed periodic orbit whose length is in the interval $(T_l - \delta_l, T_l]$, and we expect that the variances for $(m_l)$ converge to the variance of the MME, along the lines of the basic argument in \cite[Theorem 1]{mPnotes}. Thus, in negative curvature, we expect that $ \underline \sigma^2_{\text{Dyn}}(f, (m_l)) =  \sigma^2_{\text{Dyn}}(f, \mukbm)$. In negative curvature, the variance $\sigma^2_{\text{Dyn}}(f, \mukbm)$ vanishes if and only if the observable is a coboundary \cite{PP90}.  It would be interesting to characterize the class of observables for which $\underline \sigma^2_{\text{Dyn}}(f, (m_l)) =0$ in the current context, although this will require some substantial new ideas and techniques.  Although rigorous analysis of this question is beyond the scope of this paper, by analogy with the negative curvature case, our intuition is that the positive variance condition \eqref{posvar0} should be the `typical' case.

Our result extends to arrays of observable functions,  and a large class of equilibrium states. %We focus on the MME case and a single observable function to make clear the main ideas of the argument. 
Furthermore, the arguments of this paper will apply for other classes of systems with enough hyperbolicity to yield some non-uniform specification properties. We do not attempt to formalize an abstract general statement, but we hope that our proof makes clear what the roadmap should be in other related settings. We discuss these generalizations in \S \ref{ES}.

The technique is an extension of Denker, Senti and Zhang (DSZ) \cite{DSZ18}. The idea is to build $\eps$-independent collections of regular closed periodic orbits whose growth rate is the topological entropy. Classical probability theory allows us to conclude that the Lindeberg CLT holds for certain uniform measures on parameter spaces associated to these collections. The analysis of the paper relies on using the specification property to propagate this result to measures with support in $T^1M$, modeled on closed geodesics. For the analysis to work, we must restrict to closed periodic orbits with some uniform regularity. For this, we use structure provided by the work of Burns, Climenhaga, Fisher, and the first named author \cite{BCFT}. To obtain the first order behavior of measures on these orbits, we need their growth rate to be comparable to the entropy, and that there is a unique measure of maximal entropy. The first point is provided by \cite{BCFT} and the second point was originally proved by Knieper \cite{gK98}. 

While we are indebted to DSZ for the strategy and philosophy of this paper, our analysis requires several novelties. In DSZ, the focus is on discrete-time dynamical systems with uniform specification. They establish the Lindeberg CLT in their general setting, but do not explore how to verify the Lindeberg condition in examples. The novelty in the current work is that we deal with with non-uniformity and continuous-time, we apply it to geodesic flow in non-positive curvature, and we verify the Lindeberg condition from a natural positive variance condition. To achieve this, there are significant technical differences. A key difference is that our construction involves looping round closed geodesics multiple times. The reason that this is necessary is because in the flow case, it is necessary to construct the measures using segments of orbit rather than point masses. We lose independence between adjacent orbit segments due to the types of averages we are forced to consider. The looping construction is designed to compensate for this loss of independence, which is key to the whole approach.  Looping brings new technical issues - notably, the small differences in periods of the closed geodesics in $\Per^{\eta}_R(T_l - \delta_l, T_l]$ add up. This is why we require $\delta_l \to 0$, and is one reason that the choice of constants in our construction is subtle. A by-product of our construction is that it easily generalizes to the case of equilibrium states, which was not clear in DSZ. 

The paper is structured as follows. In \S \ref{background}, we recall relevant background information. In \S \ref{construction}, we describe our construction of measures from closed geodesics. In \S \ref{s.main}, we state and prove our main results.  In \S \ref{Verifylind}, we show how to check the Lindeberg condition under a suitable positive variance condition. In \S \ref{ES}, we discuss various extensions of our main results.
\section{Background} \label{background}

\subsection{Preliminaries, entropy, and pressure}
We consider a continuous flow $(g_t)$ on a compact metric space $(X, d)$. For $\eps>0$ and $t>0$, and $x \in X$, we define the dynamical (Bowen) ball to be
\[
B_t(x ,\eps) = \{ y \in X : d(f_s x, f_s y) < \eps \text{ for all } 0 \leq s \leq t\}.
\]
 For a continuous function $f: X \to \RR$, we write
\[
F(x, t) = \int_0^t f(g_\tau x) d \tau.
\]
We also write
\[
F(x, [s, t]) = F(g_sx, t-s) = \int _s^t f(g_\tau x) d \tau
\]
We use analogous notation when we use other lower case letters for an observable. Thus, for example, for an observable $h$, we write $H(x, t) = \int_0^t h(g_\tau x) d \tau$. 

We consider collections of finite-length orbit segments $\CCC \subset X \times [0, \infty)$, where $(x,t)$ is identified with the orbit segment $\{g_s x : s \in [0, t)\}$. For $t>0$, we define $\CCC_t = \{x \in X, (x, t) \in \CCC \}$. We say $E \subset Z$ is $(t, \eps)$-separated for $Z$ if for all $x, y \in E$, $y \notin \overline {B_t (x, \eps)}$. 

For $\CCC \subset X \times [0, \infty)$, the entropy $h(\CCC, \eps)$ at scale $\eps$ is defined as
\[
h(\CCC, \eps) = \limsup_{t \to \infty} \frac{1}{t} \log \sup \{ \# E : E \subset \CCC_t \text{ is } (t, \eps)\text{-separated} \},
\]
and $h(\CCC) = \lim_{\eps \to 0} h(\CCC, \eps)$. For a set $Z$, we define $h(Z, \eps)$ as $h(\CCC_Z, \eps)$, where $\CCC_Z = \{ (x, t) : x \in Z, t\in [0,\infty) \}$. In particular, $h(X, \eps)$ reduces to the standard definition of topological entropy, see \cite{pW82}. The Variational Principle states that $h(X)$ is the supremum of the measure-theoretic entropies $h_\mu$ taken over flow-invariant probability measures. A measure achieving the supremum is called a measure of maximal entropy.

%Given a potential function $\varphi: X \to \RR$, we can also define the topological pressure of the space $P(X, \varphi)$, and the topological pressure of a collection of orbit segments $P(\CCC, \varphi)$. An equilibrium state for $\varphi$ is a flow-invariant probability measure $\nu$ which satisfies $P(X, \varphi) = h_\nu + \int \varphi d \nu$. Since in this paper we mostly focus on the MME case, we do not state the definitions here, but refer to \cite{pW82} for the classic reference, and \cite{BCFT} for a recent presentation which uses our notations.

\subsection{Central Limit Theorem}
The Central Limit Theorem in dynamics describes the second order behavior of the sequence of ergodic sums/integrals. The classical CLT for a continuous flow equipped with an ergodic measure $\mu$ says that for a H\"older observable $f$ with $\int f d \mu =0$, the sequence $\frac{1}{\sqrt t} F( \cdot, t)$ converges in distribution to the normal distribution. This result was proved for hyperbolic flows by Ratner \cite{mR73}, and strengthened by Denker and Phillip in \cite{DP84}. See also Parry and Pollicott \cite{PP90}.

The classical Central Limit Theorem can be obtained using a variety of techniques. We do not attempt to survey the literature here, but we recommend recent papers by \cite{DDKN20, BM18, DSZ18, jC15, MT, SG15} for an excellent paper trail. One might expect the classical CLT to hold in the setting of this paper, but none of these proof techniques are currently known to apply. We also mention an interesting recent related result - an asymptotic central limit theorem for lengths of closed geodesics in hyperbolic surfaces was recently proved by Gekhtman, Taylor and Tiozzo \cite{GTT}.

Our result is based on the Lindeberg CLT, which is one of the most famous generalizations of the classical CLT. We recall its statement in its original context of a sequence of independent random variables. First we define the Lindeberg function for a probability measure $\nu$ and an observable $h$, and a constant $c \geq 0$.
\begin{definition} \label{Lind}
Let $Z(c) = Z(c, h, \nu) = \{x : |h - \int h d \nu| > c\}$. The \emph{Lindeberg} function is
\[
L_{\nu}(h,c):=\int (h-\int h d \nu)^2\mathbb{1}_{Z(c)}(v)d\nu(v)
\]
\end{definition}
 Recall that for a probability measure $\nu$ on a space $\Omega$ and a function $f: \Omega \to \RR$, the variance $\sigma_\nu(f)$ is defined by 
\begin{equation} \label{variance}
\sigma^2_\nu(f) = \int \left (f - \int f d\nu \right )^2 d\nu = \int f^2 d \nu - \left (\int f d \nu \right)^2.
\end{equation}

\begin{theorem}[Lindeberg CLT for independent random variables]
Let $(\Omega, \nu)$ be a probability space and let $(X_i)_{i=1}^\infty$ be an independent sequence of random variables. Let $\sigma_i$ be the variance of $X_i$, and let $s_n^2 = \sum_{i=1}^n\sigma_i^2$. Suppose that for every $\eps>0$
\begin{equation} \label{originalLind}
\lim_{n \to \infty} \frac{1}{s_n^2}\sum_{i=1}^n L_{\nu}(X_i, \eps s_n) = 0.
\end{equation}
Then $\frac{1}{s_n} \sum_{i=1}^{n}(X_i - \int X_i d \mu)$ converges in distribution to the standard normal distribution.
\end{theorem}
The hypothesis \eqref{originalLind} is  called the Lindeberg condition, see e.g. \cite{wF68}. We will formulate our results  using a dynamical version of the Lindeberg condition on periodic orbits, following Denker, Senti and Zhang \cite{DSZ18}.

\subsection{Geometry and dynamics of the geodesic flow} We recall the necessary background from \cite{BCFT} on geodesic flow for non-positive curvature manifolds. The arguments in this paper use the dynamical structure obtained there, rather than direct geometric arguments. We refer to \cite{wB95,pE99} for general geometric background.

We consider a compact, connected, boundaryless smooth manifold $M$ equipped with a smooth Riemannian metric $g$, with non-positive sectional curvatures at every point. For each $v$ in the unit tangent bundle $T^1M$ there is a unique constant speed
geodesic denoted $\gamma_v$ such that $\dot{\gamma}_v(0)=v$.  The \emph{geodesic flow} $(g_t)_{t\in\mathbb{R}}$ acts on $T^1M$  by $g_t(v)=(\dot\gamma_v)(t)$.  We equip $T^1M$ with a metric $d$ given by
\begin{equation}\label{eqn:dK}
d(v,w) = \max \{d_M(\gamma_v(t), \gamma_w(t)) \mid t \in [0,1] \},
\end{equation}
where $d_M$ is the Riemannian distance on $M$.  The flow is entropy expansive, which implies that for sufficiently small $\epsilon$, $h(X)=h(X, \eps)$. We call such a scale an \emph{expansivity constant}. Any positive $\epsilon$ which is less than one third of the injectivity radius of $M$ is an expansivity constant.

Given $v\in T^1M$, stable and unstable horospheres $H^s_v$ and $H^u_v$ can be defined locally using Jacobi fields or using a standard geometric construction in the universal cover. The horospheres are $C^2$ manifolds.  The (strong) stable and unstable manifolds $W^s_v, W^u_v$ are defined as normal vector fields to $H^s_v, H^u_v$, and we can define  the stable and unstable subspaces $E^s_v, E^u_v \subset T_vT^1M$ to be the tangent spaces of $W^s_v, W^u_v$ respectively. The weak stable manifold $W^{cs}_v$ is defined in the obvious way so that its tangent space is $E^s_v \oplus E^0_v$, where $E^0_v$ is the space given by the flow direction. The bundles $E^s, E^u$ are invariant, and depend continuously on $v$, see \cite{pE99, GW99}.

We define the \emph{singular set} $\Sing$ as the set of $v \in T^1M$ so that the geodesic determined by $v$ has a parallel orthogonal Jacobi field, and $\Reg$ to be the complement of $\Sing$.  We say that $M$ is rank one if $\Reg \neq \emptyset$. The Jacobi field formalism is used extensively in \cite{BCFT}, and we refer there for full definitons.

A key piece of geometric data which is at the heart of our analysis is a continuous function $\lambda\colon T^1M\to [0,\infty)$ defined in \cite{BCFT}. Roughly, $\lambda(v)$ is the smallest normal curvature at $v$ (with sign chosen to be non-negative) of the stable and unstable horospheres centered at $v$. If $\lambda(v)>0$, then $v \in \Reg$. We refer to \cite{BCFT} for the precise definition and more geometric context. Let $\Reg(\eta) = \{ v : \lambda(v) \geq \eta\}$. If $v \in \Reg(\eta)$, then we have various uniform estimates at the point $v$, for example on how distance scales in the local stable and unstable manifolds. These are the properties exploited in this paper. We recall the precise statement obtained on local product structure.

\begin{lemma} \cite[Lemma 4.4]{BCFT}
For every $\eta>0$, there exist $\delta>0$ and $\kappa\geq 1$ such that at every $v\in \Reg(\eta)$, the foliations $W^u, W^{cs}$ have local product structure with constant $\kappa$ in a $\delta$-neighborhood of $v$. That is,  for every $\eps\in (0,\delta]$ and all $w_1,w_2 \in B(v,\eps)$, the intersection $W_{\kappa \eps}^u(w_1) \cap W_{\kappa \eps}^{cs}(w_2)$ contains a single point, which we denote by $[w_1,w_2]$, and
\begin{align*}
d^u(w_1, [w_1,w_2]) &\leq \kappa d(w_1,w_2), \\
d^{cs}(w_2, [w_1,w_2]) &\leq \kappa d(w_1,w_2).
\end{align*}
\end{lemma}

Uniformity of the local product structure on $\Reg(\eta)$ is used to obtain the specification property for orbit segments starting and ending in $\Reg(\eta)$. Precisely, we define the collection of orbit segments
 \[
 \CCC (\eta):=\{(v,t): \lambda(v) \geq\eta, ~ \lambda(g_t v) \geq \eta \}.
 \]
We have the following result.
\begin{theorem} \cite[Theorem 4.1]{BCFT}
For each $\eta>0$, the collection of orbit segments $\CCC(\eta)$  has the specification property. That is,  given $\rho>0$,  there exists $\tau=\tau(\rho)$ such that for every $(x_1, t_1)$, $\dots, (x_N, t_N)\in \mathcal{C}$  and \emph{every} collection of times $\tau_1, \ldots, \tau_{N-1}$ with $\tau_i \geq \tau$ for all $i$, there exists a point $y\in X$ such that for $s_0=\tau_0=0$ and $s_j=\sum_{i=1}^j t_i + \sum_{i=0}^{j-1} \tau_i$, we have
\[
f_{s_{j-1} + \tau_{j-1}}(y)\in B_{t_j}(x_j, \rho)
\]
for every $j\in \{ 1,\dots, N\}$.  
\end{theorem}
We recall some other results that we will use from \cite{BCFT} and \cite{CT16}. We often consider the following set of orbit segments
\[
\BBB(\eta):=\{(v,t): \frac{\int_{0}^t \lambda(g_u(v))du}{t}<\eta\}.
\]
Note that $\lambda$ vanishes on $\Sing$, so any orbit segment in $\Sing \times [0, \infty)$ is a member of $\BBB(\eta)$. It was shown in \cite[\S 5]{BCFT} that $\lim_{\eta \to 0} h(\BBB(\eta)) = h(\Sing)$. For the class of geodesic flows under consideration, it is known that 
\[
h(\Sing) < h(T^1M). 
\] 
This is easy in the case that $M$ is a surface, since $h(\Sing)=0$. However, this entropy gap is a highly non-trivial fact in higher dimensions. It was first proved as a consequence of Knieper's work \cite{gK98}, and a direct proof is given in \cite{BCFT}. The geodesic flow has a unique measure of maximal entropy, known as the Knieper-Bowen-Margulis measure, which we denote by $\mu_{\text{KBM}}$.

 \subsection{Counting closed regular geodesics} \label{countingclosed}
For a small $\delta>0$, we define $\Per_R(T-\delta, T]$ to be the set of closed regular geodesics which have length in the interval $(T-\delta, T]$. For $\gamma \in \Per_R(T-\delta, T]$, we write $|\gamma|$ for its length, and $v_\gamma$ for an element of $T^1M$ chosen to be tangent to $\gamma$.

 Recall from Proposition 6.4 in \cite{BCFT}, for any $\delta>0$, there exists $T_{\delta}>0$ and
 \begin{equation} \label{betabound}
     \begin{aligned}
     \beta=\beta(\delta)\approx e^{-hT_{\delta}}
     \end{aligned}
 \end{equation}
such that for all $T>T_{\delta}$, we have
\begin{equation}
    \begin{aligned}
    \frac{\beta}{T}e^{Th}\leq  \# \Per_R(T-\delta, T] \leq \beta^{-1}e^{Th}.
    \end{aligned}
\end{equation}
We take $T_\delta \to \infty$ as $\delta \to 0$ (and the proofs of Proposition 4.5, Lemma 4.7 and Proposition 6.4 in \cite{BCFT} show that this is necessary). For $\eta>0$, we define the \emph{uniformly regular closed geodesics} as
 \[
 \text{Per}_R^{\eta}(T-\delta,T]:=\{ \gamma \in \Per_R(T-\delta, T] : \int_0^{|\gamma|} \lambda (g_s v_\gamma) ds \geq |\gamma| \eta \},
 \]
 that is the collection of elements in $\text{Per}_R(T-\delta,T]$ whose average of $\lambda$ is at least $\eta$. Writing $h':=\frac{h(\text{Sing})+h}{2}$, we fix $\eta>0$ throughout the rest of the paper such that $h(\BBB(2\eta))< h'<h$. We also choose $\epsilon$ so that $4\epsilon$ is an expansivity constant. In particular, $h(T^1M, 4\epsilon)=h$. Notice that we can choose $\epsilon$ smaller if necessary. 

Define $\delta':=\frac{\eta}{\lambda_{\text{max}}}$ where $\lambda_{\text{max}}:=\max\{\lambda(v):v\in T^1M\}$. We now argue that for $\delta$ sufficiently small, $\# \Per_R^{\eta}(T-\delta,T]$ is bounded uniformly from below.
\begin{lemma} \label{T0}
For any $\delta<\delta'$, there exists $T_0=T_0(\delta,\eta)$ such that for all $T>T_0$,
\begin{equation} \label{eqlowerPer}
    \begin{aligned}
     \# \Per_R^{\eta}(T-\delta,T]\geq \frac{\beta}{2T}e^{Th}.
    \end{aligned}
\end{equation}
\end{lemma}
\begin{proof}
Recall that we fix $\eta$ such that $h(\BBB(2\eta)) <h'$. It follows that there exists $T_0'=T_0'(\eta)>0$ so for all $T>T_0'$, there are maximal $(T, 4 \epsilon)$-separated sets $E_T$ for $\BBB(2 \eta)$ so that $\#E_T< e^{Th'}$ and also so that $e^{Th'}< \frac{\beta}{2T}e^{Th}$. Given $\delta\in (0,\delta')$, define $T_0(\delta,\eta):=\max\{T_0'(\eta),T_{\delta},1\}$.  With a fixed $\eta$, since $T_\delta \to \infty$ as $\delta \to 0$, we observe that $T_0 (\delta,\eta)=T_\delta$ when $\delta$ is sufficiently small. We write $\Per_R^{< \eta}(T-\delta,T] := \Per_R(T-\delta,T] \setminus \Per_R^{\eta}(T-\delta,T]$.  For $T>T_0$ and any $\gamma\in \text{Per}_R^{<\eta}(T-\delta,T]$, we choose a vector $v_{\gamma}\in T^1M$ such that it is tangent to $\gamma$ at some point. Due to the difference in the period of elements in $\text{Per}_R(T-\delta,T]$, different choices of $v_{\gamma}$ may lead to variations in the precise value of $\int_0^T\lambda(g_sv_{\gamma})ds$. However, we have
\begin{equation}         \label{lambda2eta}
    \begin{aligned}
    \int_0^T\lambda(g_s v_{\gamma})ds\leq |\gamma|\eta+\delta'\lambda_{\text{max}}<T\eta+\eta<2T\eta,
    \end{aligned}
\end{equation}
which shows that we always have $(v_{\gamma},T)\in \BBB(2\eta)$. By the choice of $4\epsilon$ and $\mathsection 6$ in \cite{gK98} we know elements in $\text{Per}_R(T-\delta,T]$ are $(T,4\epsilon)$-separated, which in turns shows that $\Per_R^{< \eta}(T-\delta,T]< e^{Th'}<\frac{\beta}{2T}e^{Th}$. As a consequence, we have
 \[
    \begin{aligned}
    \# \Per_R^{\eta}(T-\delta,T] & = \# \Per_R^{\eta}(T-\delta,T]-\# \Per_R^{<\eta}(T-\delta,T] \\
    & >\frac{\beta}{T}e^{Th}-\frac{\beta}{2T}e^{Th}=\frac{\beta}{2T}e^{Th}. \qedhere
    \end{aligned}
\]
\end{proof}

From now on we always assume that $\delta$ and $T$ satisfy the conditions in Lemma \ref{T0}. By the definition of $\Per_R^{\eta}(T-\delta,T]$, if $\gamma$ is an element in $\Per_R^{\eta}(T-\delta,T]$, there must be some $t\in [0,T)$ such that $v = \dot \gamma(t) \in \text{Reg}(\eta)$. Since $v$ is periodic, we know that $(v, |\gamma|)\in \CCC(\eta)$. 

For each $\gamma \in \text{Per}_R^{\eta}(T-\delta,T]$, we choose $v=v(\gamma)$ such that $v \in \gamma \cap \Reg(\eta)$. We let 
\[
E_{\delta}(T)= \{v(\gamma) : \gamma \in \text{Per}_R^{\eta}(T-\delta,T]\},
\]
recalling that $E_{\delta}(T)$ is a $(T, 4\epsilon)$-separated set. From the definition of $E_{\delta}(T)$ and \eqref{eqlowerPer}, we know that
\begin{equation}      \label{ETlowerbound}
    \begin{aligned}
     \# E_{\delta}(T) =   \# \Per_R^{\eta}(T-\delta,T] \geq \frac{\beta}{2T}e^{Th}.
    \end{aligned}
\end{equation}
We will often work with the collection of orbit segments
\[
\{(v, T) : v \in E_{\delta}(T) \}.
\]
Here we use the same $T$ across all $v\in E_{\delta}(T)$ so we can compare lengths uniformly - note that $T$ differs from the least period of $\gamma(v)$ by at most $\delta$. 

\subsection{Growth of variations on $\CCC(\eta)$}

For a collection of orbit segments $\CCC$, any $\delta>0, T>0$ and $h\in C(T^1M)$ we define
 $$\omega(h,T, \delta,\CCC):=\sup_{(u,T)\in \CCC ,v\in B_T(u,\delta)} \left |H(u,T)-H(v,T) \right |.$$

The following analogy of Lemma 5.6 in \cite{TW19} holds for $\omega$, and is a crucial estimate in the construction given in \S \ref{construction}.

\begin{lemma} \label{refscale}
Let $\CCC(\eta)$ be the collection previously defined. Then for sufficiently small $\delta_0$, for any $h\in C(T^1M)$, we have 
\begin{equation}
    \lim_{T\rightarrow \infty}\frac{\omega(h,T, \delta_0,\CCC(3\eta/4))}{T}=0.
\end{equation} 
\end{lemma}

\begin{proof}
This proof is parallel to the one of Lemma 5.6 in \cite{TW19}. Choose the same $\eta$ as before and $\delta_0$ such that 
\begin{enumerate}
    \item $\text{Reg}(\frac{3\eta}{4})$ has local product structure at scale $4\delta_0'$, with coefficient $\kappa=\kappa(\frac{3\eta}{4},4\delta_0')>1$. Take $\delta_0=\delta_0'/\kappa$.
    \item for any $u,v\in T^1M$ such that $d(u,v)<\kappa\delta_0$, we have $|\lambda(u)-\lambda(v)|<\frac{\eta}{4}$. In particular, we have $B(\text{Reg}(\frac{3\eta}{4}),\kappa\delta_0)\subset \text{Reg}(\frac{\eta}{2})$ and $B(\text{Reg}(\frac{\eta}{2}),\kappa\delta_0)\subset \text{Reg}(\frac{\eta}{4})$.
  
\end{enumerate}

%We can choose such $\delta_0$ by  first fixing some scale $\delta_0''$ that (2) holds, then choosing $\delta_0'$ smaller than $\delta_0''$ such that the local product structure holds and defining $\delta_0$ accordingly. 
Consider $(u,T)\in \CCC(3\eta/4)$ and $v\in B_T(u,\delta_0)$. By the local product structure, there is a vector $u_0\in T^1M$ such that $u_0\in W^s_{\kappa \delta_0}(u)\cap W^{cu}_{\kappa \delta_0}(v)$. Then there exists $s\in (-\kappa\delta_0,\kappa\delta_0)$ such that $g_s(u_0)\in W^u_{\kappa \delta_0}(v)$. Observe that $d(g_Tu ,g_Tv )<\delta_0$, $d(g_Tu,g_T u_0)<\kappa\delta_0$, $d(g_T u_0,g_{T+s} u_0)<\kappa\delta_0$. Combining the above three inequalities together we have $d(g_T v,g_{T+s}u_0)<3\kappa\delta_0$. As $g_{T+s}u_0\in W^u(g_Tv)$, we know $d^u(g_{T+s}u_0,g_Tv)<3\kappa^2\delta_0$, therefore $d^{cu}(g_T u_0,g_T v)<4\kappa^2\delta_0=4\kappa\delta_0'$. In other words, $g_T u_0\in W^s_{4\kappa \delta_0'}(g_T u)\cap W^{cu}_{4\kappa \delta_0'}(g_T(v))$. As $d(g_T u,g_T v)<\delta_0$ and $g_T u\in \text{Reg}(\frac{3\eta}{4})$, by the local product structure we know $g_T u_0\in W^s_{\kappa \delta_0}(g_T u)\cap W^{cu}_{\kappa \delta_0}(g_T v)$. In particular, $g_{T+s} u_0\in W^u_{\kappa \delta_0}(g_T v)$. We conclude that $g_{t+s} u_0\in W^u_{\kappa \delta_0}(g_t v)$ and $g_{t} u_0\in W^s_{\kappa \delta_0}(g_t u)\cap W^{cu}_{\kappa \delta_0}(g_t v)$ for all $t\in [0,T]$.  Now, for any fixed $h\in C(T^1M)$, we can bound the variation of $h$ over $(u,T)$ and $(v,T)$ by variations along the stable, central and unstable directions. To be more precise, we have
$|H(u, T)- H(v, T)| \leq |H(u, T)-H(u_0, T)| +| H(u_0, T) - H (g_{s} u_0, T) |+|H (g_{s} u_0, T)- H(v, T)|$. From the definition of $\CCC(3\eta/4)$ and property (2) of $\delta_0$, we know $\lambda(u),\lambda(g_T u)>\frac{3\eta}{4}$ and $\lambda(v),\lambda(g_T v)>\frac{\eta}{2}$, so $(v,T)\in \CCC(\frac{\eta}{2})$. Therefore, to prove (2.1), it suffices to prove the following 
\begin{equation}    \label{omegas}
    \lim_{T\rightarrow \infty}\frac{\omega_s(h,T;\kappa\delta_0,3\eta/4)}{T}=0
\end{equation}

and
\begin{equation}    \label{omegau}
    \lim_{T\rightarrow \infty}\frac{\omega_u(h,T;\kappa\delta_0,\eta/2)}{T}=0,
\end{equation}
where 
$$
\omega_s(h,T;\kappa\delta_0,3\eta/4):=\sup_{g_T(u)\in \text{Reg}(3\eta/4) ,v\in W^s_{\kappa\delta_0}(u)}|H(u, T) - H(v, T)|
$$
and
$$
\omega_u(h,T;\kappa\delta_0,\eta/2):=\sup_{u\in \text{Reg}(\eta/2) ,v\in g_{-T}(W^u_{\kappa\delta_0}(g_T(u)))}|(H(u, T)-H(v, T)|.
$$

Let us prove \eqref{omegas}. Consider $u'\in T^1M$ such that $g_T u'\in \text{Reg}(3\eta/4)$ and $v'\in W^s_{\kappa\delta_0}(u')$. For any $\hat{\epsilon}>0$, by uniform continuity of $h$ on $T^1M$ we know there exists $\hat{\delta}>0$ such that if $v_1,v_2\in T^1M$, $d(v_1,v_2)<\hat{\delta}$, then $|h(v_1)-h(v_2)|<\hat{\epsilon}$. Meanwhile, property (2) of $\delta_0$ shows that any vector $\hat{v}$ lying on the local stable arc connecting $u'$ and $v'$  satisfies $\lambda(g_T \hat{v})>\frac{\eta}{2}$. Following the proof of Lemma 3.10 in \cite{BCFT}, for any $0\leq t_1\leq t_2\leq T$ we have 
\begin{equation}    \label{stablecontraction}
    \begin{aligned}
    d^s(g_{t_1} u',g_{t_1} v')\geq e^{\eta(t_2-t_1)/2}d^s(g_{t_2} u',g_{t_2} v'). 
    \end{aligned}
\end{equation}

Since $d^s(u',v')<\kappa\delta_0$, by \eqref{stablecontraction} we have $d^s(g_t u',g_t v')<\kappa\delta_0e^{-\eta t/2}$. By writing $(2\log(\frac{\kappa\delta_0}{\hat{\delta}}))/\eta$ as $\hat{T}$ and assuming that $T>\hat{T}$ (which is possible since the choice on $\hat{T}$ does not depend on $T$ and $T$ approaches $\infty$), it is easy to see that $d(g_t u',g_t v')\leq d^s(g_t u',g_t v')<\hat{\delta}$ for $t\in [\hat{T},T]$. Therefore, we have $|H(u',T)-H(v',T)|
    \leq |H(u',\hat{T})-H(v',\hat{T})|+|H(g_{\hat{T}}u',T-\hat{T})-H(g_{\hat{T}}v',T-\hat{T})| 
    \leq 2\hat{T}||h||+(T-\hat{T})\hat{\epsilon}$, and this holds for all such $u',v'$ and $T>\hat{T}$, which shows that $\lim_{T\rightarrow \infty}\frac{\omega_s(h,T;\kappa\delta_0,3\eta/4)}{T}\leq \hat{\epsilon}$. By making $\hat{\epsilon}$ arbitrarily small, \eqref{omegas} is proved. \eqref{omegau} is proved similarly by replacing $3\eta/4,\eta/2$ with $\eta/2,\eta/4$. %This concludes the proof of Lemma \ref{refscale}.
\end{proof}
The small $\delta_0$ in the above lemma can be chosen so that $\epsilon<\delta_0<\delta'$, and we will do so in \S \ref{construction}. Recall $\epsilon$ is a choice of scale so that $4 \epsilon$ is an expansivity constant, and can be chosen arbitrarily small, so we can ensure that it is chosen smaller than $\delta_0$.

\section{Construction of measures} \label{construction}

We now construct sequences of measures that converge to $\mukbm$, and are reference measures for our CLT. Recall $T_0$ and $\delta_0$ are chosen as in Lemma \ref{T0} and Lemma \ref{refscale}. We start by constructing a sequence of $4$-tuples $(T_l,k_l,\delta_l,C_l)_{l\in \mathbb{N}}$ as follows.
\begin{hypothesis} \label{condstar}
We choose sequences $T_l \in (0, \infty)$, $k_l \in \mathbb{N}$, $\delta_l \in (0, \delta_0)$, and $C_l \in \NN$  which satisfy the following relationships:

1)   For all $l\in \mathbb{N}$, $T_l>\max\{T_0(\delta_l,\eta),1\}$

2) $T_l \uparrow \infty$, $\frac{T_l}{T_0(\delta_l,\eta)} \uparrow \infty$ and $k_l \uparrow \infty$

3)  $k_l\delta_l^2\downarrow 0$ and $\frac{\sqrt{k_l}T_l}{C_l}\downarrow 0$. 
\end{hypothesis}
Sequences which satisfy these conditions can be easily found by first choosing $\delta_l$, then $T_l$, then $k_l$ and $C_l$. For each $l$, let
\[
E_l := E_{\delta_l}(T_l)
\] 
be a $(T_l, 4\epsilon)$-separated set chosen by following the procedure described in \S \ref{countingclosed}. Each $x \in E_l$ corresponds to a regular closed geodesic $\gamma(x)$ with least period in the interval $(T_l-\delta_l, T_l]$. We write $t=t(x)$ for the period of $\gamma(x)$, and we recall that by construction $(x,t) \in \CCC(\eta)$.

For each $l\in \mathbb{N}$, we consider $E_l^{k_l}$, which is the Cartesian product of $E_l$ of order $k_l$. By the specification property on $\CCC(\eta)$ at scale $\epsilon$, we define a sequence of maps $\{\pi_l\}_{l\in \mathbb{N}}$ $: E_l^{k_l}\rightarrow T^1M$ as follows. The map $\pi_l$ sends $\ul x=(x_1,x_2,\dots,x_{k_l})$ to $\pi_l( \ul x)$ by finding a point which tracks the periodic orbit defined by $x_1$ for $C_l$ times, and then tracks the periodic orbit defined by $x_2$ for $C_l$ times, etc. The transition times (which depend on the choice of $\ul x$) are chosen so that times line up correctly at the start of each prescribed periodic orbit, independent of the choice of $\ul x$. 

More precisely, let $\ul x = (x_1, x_2, \ldots, x_{k_l}) \in E_l^{k_l}$. Since each $(x_i, t_i)$ with $x_i \in E_l$ is a member of $\CCC(\eta)$, each such orbit segment has the specification property at scale $\epsilon$. We use this property to construct a point $z= \pi_l(\ul x)$ such that
\begin{enumerate}
\item $d_{C_lt_1}(z, x_1)< \epsilon$,
\item $d_{C_lt_2}(g_{C_lT_l+M} z, x_2) < \epsilon$,
\item $d_{C_lt_3}(g_{2(C_lT_l+M)} z, x_3) < \epsilon$,
\end{enumerate}
and continue this way so that
$$d_{C_lt_i}(g_{(i-1)(C_lT_l+M)}z,x_i)<\epsilon$$
for all $1\leq i \leq k_l$. In the above, $M=M(\eta, \epsilon)$ is the transition time in specification for $\CCC(\eta)$. We note that the transition time between looping around one periodic orbit to the next is bounded by $M$ from below and $C_lT_l-C_l(T_l-\delta_l)+M=C_l\delta_l+M$ from above. We define
\[
P_l = \pi_l(E_l^{k_l}).
\]
Since $E_l$ is $(T_l, 4\epsilon)$-separated and we are applying specification at scale $\epsilon$, for $\ul x=(x_1,x_2,\dots,x_{k_l})$ and  $\ul y =(y_1,y_2,\dots,y_{k_l})\in (E_l)^{k_l}$, we have
\begin{enumerate}
    \item If $x_1=y_1$, then $d_{C_l(T_l-\delta_l)}(\pi_l(\ul x),\pi_l(\ul y))<2\epsilon$,
    \item If $x_1 \neq y_1$, then $d_{T_l}(\pi_l(\ul x),\pi_l(\ul y))>2\epsilon$,
\end{enumerate}
and similarly for each $i \in \{2, \ldots, k_l\}$. In particular, $\# P_l = \# E_l^{k_l}$ and the set $P_l$ is $(k_lC_lT_l + (k_l-1)M, 2\epsilon)$-separated. 

We define a measure $m_l$ by uniformly distributing mass over $E_l$, i.e. we let 
\[
m_l=\frac{1}{\#E_l}\sum_{v\in E_l}\delta_{v}.
\]
Now define $\mu_l$ to be the self-product measure of $m_l$ on $E_l^{k_l}$, equivalently the uniform measure on $E_l^{k_l}$, which is written as 

 \[
 \mu_l:=\frac{1}{\#E_l^{k_l}}\sum_{\ul x \in E_l^{k_l}}\delta_{\ul x}.
 \]
We write $L(v, t)$ for the natural measure along the orbit segment $(v, t)$, in the sense that for any continuous function $\phi$, $\int \phi ~d L(v, t) = \int_0^t \phi(g_s v) ds$. For each $l$, define a sequence of probability measures $\nu_l$ on $T^1M$ by
\[
\nu_l= \frac{1}{\# P_l}\sum_{y \in P_l} \frac{1}{T_l} L(y, T_l) = \frac{1}{\# E_l^{k_l}}\sum_{\ul x\in E_l^{k_l}} \frac{1}{T_l} L(\pi_l( \ul x), T_l).
\]
Note that although $\nu_l$ puts mass along only the orbit segment of length $T_l$, we will be interested in evaluating potentials of the form $F(\cdot, [s,t])$ with respect to $\nu_l$ with $s< t$ taking carefully chosen values in the interval $[0, k_lC_lT_l + (k_l-1)M]$. Integrals of these functions thus incorporate information along the whole prescribed length of the orbit segment. We often state our results with time running up to $k_l(C_l T_l + M)$, since this is a slightly simpler expression and the extra run of time $M$ makes no difference. This is the time $S_l$ denoted in Theorem \ref{intromain}.
\begin{lemma}   \label{nulconvergence}
Given $(T_l,k_l,\delta_l,C_l)_{l\in \mathbb{N}}$ satisfying Hypothesis \ref{condstar} the corresponding sequence of measures $\nu_l$ converges to the measure of maximal entropy $\mukbm$.
\end{lemma}

\begin{proof}
It is convenient to define another sequence of probability measures on $T^1M$, $\{\mu_l^{*}\}_{l\in \mathbb{N}}$ by
$\mu_l^{*}=\frac{1}{\# E_l}\sum_{v \in E_l}\frac{1}{T_l} L(v, T_l)$. We show that $\mu_l^*$ converges to $\mu_{\text{KBM}}$ when $l\rightarrow \infty$. It is not hard to observe that any $\text{weak}^*$ limit measure of $\mu_l^*$ is $g_t$-invariant for all $t$ since $\delta_l\downarrow 0$ and $T_l\uparrow \infty$. Recall that $T_{\delta_l}$ is the time defined at \eqref{betabound}. We know $E_l$ is $(T_l, 2\epsilon)$-separated, $h(2\epsilon)=h$ and 
$$
\begin{aligned}
\liminf_{l\rightarrow \infty}\frac{1}{T_l}\log \#E_l&\geq \lim_{l\rightarrow \infty}\frac{1}{T_l}\log(\frac{\beta(\delta_l)}{2T_l}e^{T_lh})  \geq \lim_{l\rightarrow \infty}\frac{1}{T_l}\log(\frac{e^{-T_{\delta_l}h}}{2CT_l}e^{T_lh}).
\end{aligned}
$$ 
Observe from the proof of Lemma \ref{T0} that $T_0(\delta_l,\eta)= T_{\delta_l}$ for $\l$ sufficiently large. Thus by Hypothesis \ref{condstar}, for any small $\epsilon'>0$, $T_{\delta_l}/T_l<\epsilon'$ for large enough $l$, so $e^{-T_{\delta_l}h}> e^{-\epsilon'T_lh}$. Thus
\[
\lim_{l\rightarrow \infty}\frac{1}{T_l}\log(\frac{e^{-T_{\delta_l}h}}{2CT_l}e^{T_lh}) \geq \lim_{l\rightarrow \infty}\frac{1}{T_l}\log(\frac{e^{(1-\epsilon')T_lh}}{2CT_l})=(1-\epsilon')h.
\]
From the choice of $\epsilon'$, it follows that $\lim_{l\rightarrow \infty}\frac{1}{T_l}\log \#E_l  = h$. The proof of the second half of variational principle in \cite{pW82} will imply that $\mu_l^{*}$ converges to $\mu_{\text{KBM}}$ in the $\text{weak}^*$-topology. Therefore, to prove the statement in the lemma, it suffices to show that for any $f\in C(T^1M)$, we have $\lim_{t\rightarrow \infty}\int fd\mu_{l}^*=\lim_{t\rightarrow \infty}\int fd\nu_l$.

Notice that for a fixed $x_1\in E_l$ and any $(x_2, \ldots, x_{k_l}) \in E_l^{k_l-1}$, we have 
\begin{equation} \label{convergencelemma1}
|F(\pi_l(x_1, \ldots, x_{k_l}), T_l) - F(x_1, T_l)| \leq \omega(f, T_l, \eps, \CCC(3\eta/4)).
\end{equation}
Here the scale $3\eta/4$ is by property (2) in the choice of $\delta_0$.

Averaging over all $(x_2, \ldots, x_{k_l})$, we have
\begin{equation}  \label{convergencelemma2}
    \begin{aligned}
      &\left | \frac{1}{(\# E_l)^{k_l-1}} \sum_{x_2, \ldots, x_{k_l}} (F( \pi_l(x_1, x_2, \ldots, x_{k_l}), T_l) - F(x_1, T_l)) \right | \\
      &\leq \frac{1}{(\# E_l)^{k_l-1}} \sum_{x_2, \ldots, x_{k_l}} \left | F( \pi_l(x_1, x_2, \ldots, x_{k_l}), T_l) - F(x_1, T_l) \right | \\
    &\leq \frac{\omega{(f,T_l, \eps, \CCC(3 \eta/4))}}{T_l}
    \end{aligned}
\end{equation}
where the second inequality follows from \eqref{convergencelemma1}. On the other hand, it is not hard to show that
\begin{equation} \label{convergencelemma3}
\begin{aligned}
\lim_{l\rightarrow \infty}|\int fd\mu_{l}^*-\int fd\nu_l| 
=\lim_{l\rightarrow \infty}\left |\frac{1}{\#E_l}\sum_{x_1\in E_l} V_l(x_1) \right|
\end{aligned}
\end{equation}
where the variation term $V_l(x_1)$ is defined as

$$V_l(x_1):=\frac{1}{(\# E_l)^{k_l-1}}  \sum_{x_2, \ldots, x_{k_l}} (F( \pi_l(x_1, x_2, \ldots, x_{k_l}), T_l) - F(x_1, T_l)).
$$
By \eqref{convergencelemma2}, for each $x_1 \in E_l$ we have $|V_l(x_1)|\leq \frac{\omega{(f,T_l, \eps, \CCC(3 \eta/4))}}{T_l}$. By plugging this into \eqref{convergencelemma3}, observing $\epsilon<\delta_0$ by property (3) of $\delta_0$ and applying Lemma \ref{refscale}, we get
$$
\lim_{l\rightarrow \infty}|\int fd\mu_{l}^*-\int fd\nu_l|\leq  \lim_{l\rightarrow \infty} \frac{\omega{(f,T_l, \eps, \CCC(3 \eta/4))}}{T_l}=0,
$$
which concludes the proof of the lemma.
\end{proof}

\subsection{Variance}
Given a function $f\in C(T^1M)$, we consider $F( \cdot, T_l): E_l \to \RR$ defined in the obvious way, i.e. $F(v, T_l) = \int_0^{T_l} f(g_t v) dt$ for $v \in E_l$, and we consider the variances
\[
\sigma^2_l := \sigma^2_{m_l} (F(\cdot, T_l)) = \frac{1}{\# E_l}\sum_{x \in E_l}\left (F(x,T_l)-\frac{1}{ \#E_l}\sum_{x \in E_l} F(x,T_l) \right )^2.
\]
Terms of the form $k_l \sigma_l^2$  appear in our version of the Lindeberg condition.
We let 
\[
Q_l:=\left \lfloor \frac{(T_l-\delta_l)C_l}{T_l} \right \rfloor-1.
\] 
The interpretation of this constant is that it is chosen so that if we spend  $Q_l T_l$ time looping around one of the closed geodesics then we have definitely not exceeded $C_l$ times the actual length of the geodesic, which is the time at which we move on to approximating the next closed geodesic.  For fixed $l\geq 2$ and each $p\in [1,k_l]$, we let $t_p = (p-1)(C_lT_l+M)$. For $q\in [0,Q_l-1]$, we define a family of function $F^l_{p,q}$ by averaging $f$ over the time interval $[t_p + qT_l, t_p +(q+1) T_l]$, that is
\[
F^l_{p,q}(v)  : = F(g_{t_p}v, [qT_l, (q+1) T_l]) = F(g_{t_p+qT_l}v, T_l) = \int_{t_{p}+q T_l}^{t_p + (q+1)T_l} f(g_t v) dt.
\]
We also consider this function summed over the range of $q$. Note that
\[
\sum_q F^l_{p,q}(v) = \sum_{q=0}^{Q_l-1}  \int_{t_{p}+q T_l}^{t_p + (q+1)T_l} f(g_t v) dt = \int_{t_{p}}^{t_p + Q_lT_l} f(g_t v) dt.
\]
Thus, we define $F^l_p$ by averaging $f$ over the time interval $[t_p, t_p+Q_lT_l]$, that is
\[
 F^l_p := F(g_{t_p}v, Q_lT_l)
 \]
We define 
\[
s^2_l=\sum_{p} \sigma^2_{\nu_l}( F_{p}^{l}) = \sum_{p} \sigma^2_{\nu_l}(\sum_q F_{p,q}^{l}),
\]
where $1\leq p\leq k_l$, $0\leq q\leq Q_l-1$. This quantity is the relevant variance quantity for the measures $\nu_l$, recording the sum of variances of each prescribed closed geodesic for $Q_l$ times. To emphasize what goes into this variance quantity, we observe that it can be easily computed that 
\[
\begin{aligned}
 \sigma^2_{\nu_l}( F_{p}^{l}) & =  \frac{1}{ \#P_l} \sum_{y \in P_l} \frac{1}{T_l}  \int_0^{T_l} \left (\int_{t_p}^{t_p+Q_lT_l} f(g_{s+t}y)ds \right )^2 dt \\ &- \left (\frac{1}{ \#P_l} \sum_{y \in P_l} \frac{1}{T_l} \int_0^{T_l} \int_{t_p}^{t_p+Q_lT_l} f(g_{s+t} y) ds dt \right )^2,
\end{aligned}
\]
and summing the above expression over $p$ from $1$ to $k_l$ gives $s_l^2$.

\subsection{Basic estimates}
We have the following comparison between averages along the total number of loops around a fixed $x_i$, and the corresponding orbit segment along $\pi((x_1, \ldots, x_{k_l}))$.

\begin{lemma} \label{bestimate1}
For a H\"older continuous potential function $f$, any $\ul x \in E_l^{k_l}$, any $t\in [0,T_l]$ and $p \in \{1, \ldots, k_l\}$, there exists $K=K(f)$ such that
\[
|F^l_p (g_t(\pi_l(\ul x))) - Q_l F( x_{p}, T_l) | \leq 2 KT_l + (2\kappa \eps +  2 \delta_l Q_l) \| f \|.
\]
\end{lemma}
\begin{proof}
Let $\ul x = (x_1, \ldots, x_{k_l}) \in E_l^{k_l}$, and let $z = \pi_l(\ul x)$, and we fix $1\leq p \leq k_l$. By construction, we have
\[
d_{C_lt(x_p)}(g_{(p-1)(C_lT_l+M)}z,x_{p})<\epsilon.
\] 
Recall that for each $x \in E_l$, we have $(x, t(x)) \in \CCC(\eta)$, where $t(x) \in [T_l-\delta_l, T_l]$ is the least period of the periodic orbit defined by $x$. 
In particular, since each lap round such a periodic orbit will carry a definite amount of hyperbolicity, the distance between the orbit of $(g_{(p-1)(C_lT_l+M)}z, C_lT_l)$ and $(x_{p},C_lT_l)$ is much smaller than $\epsilon$ when $l$ is large. This is the key idea in getting the desired estimate.

More precisely, for each $l\geq 2$, $1\leq p \leq k_l$ and $\ul x \in E_l^{k_l}$, following the proof of Lemma \ref{refscale} we know there is some $u_{p}=u_{p}(\ul x)$ such that $u_{p}\in T^1M$ and 

$$
g_{s} u_{p} \in W^s_{\kappa \epsilon}(g_{s+t_p}z\cap W^{cu}_{\kappa \epsilon}(g_{s} x_{p})
$$
for all $s\in [0,C_lt_{p+1}]$. In particular, it holds for all $s\in [0,(Q_l+1)T_l]$ by the definition of $Q_l$. There is $s_{p}=s_{p}(\ul x)$ such that $s_{p}\in [-\kappa \epsilon,\kappa \epsilon]$ and $g_{s_{p}+s} u_{p}\in W_{\kappa \epsilon}^u(g_{s} x_{p})$. Fix such $l$, $p$ and for any $t\in [0,T_l]$ we want to control 
\[
|F_{p}^l(g_tz)-Q_lF(x_{p},T_l)|,
\] 
which is bounded above by the sum of the following four terms

\begin{enumerate}
\item  $|F_{p}^l(g_tz)-F(g_tu_{p},Q_lT_l)|$
\item $|F(g_tu_{p},Q_lT_l)-F(g_{t+s_{p}}u_{p},Q_lT_l)|$
  \item $|F(g_{t+s_{p}}u_{p},Q_lT_l))-F(g_tx_{p},Q_lT_l)|$
  \item $|F(g_tx_{p},Q_lT_l)-Q_lF(x_{p},T_l)|$.
 \end{enumerate}
Let us analyze these four terms. We begin with the first term. Suppose $f$ satisfies $|f(x)-f(y)|\leq L_0 d(x,y)^{\alpha}$. We know for any $u,v \in T^1M$, $|\lambda(u)-\lambda(v)|<\frac{\eta}{4}$ whenever $d(u,v)<\delta_0$. Therefore, for each $0\leq q \leq Q_l-1$, by Lemma 3.10 in \cite{BCFT} we have
\[
\begin{aligned}
|F_{p, q}^l(g_tz)-F(g_tu_{p},[qT_l,(q+1)T_l])|
&\leq \int_{qT_l}^{(q+1)T_l}|f(g_{t+s+t_p}z)-f(g_{t+s}u_{p})|ds \\
&\leq L_0\int_{qT_l}^{(q+1)T_l} d(g_{t+s+t_p}z,g_{t+s}u_{p})^{\alpha}ds \\
&\leq L_0T_ld(g_{qT_l+t+t_p}z,g_{qT_l+t}u_{p})^{\alpha} \\ &\leq  L_0T_l\kappa\epsilon e^{-\frac{qT_l\eta\alpha}{2}}.
\end{aligned}
\]
We obtain
\[
    \begin{aligned}
    &|\sum_{q=0}^{Q_l-1}F_{p,q}^l(g_tz-F(g_tu_{p},Q_lT_l)| \leq \sum_{q=0}^{Q_l-1}L_0T_l\kappa\epsilon e^{-\frac{qT_l\eta\alpha}{2}} \leq \frac{L_0T_l\kappa \epsilon}{1-e^{-\frac{T_l\eta\alpha}{2}}} \leq KT_l,
    \end{aligned}
\]
where $K:=\frac{L_0\kappa \epsilon}{1-e^{-\frac{\eta \alpha}{2}}}$ is a constant (Recall that $T_l>1$ for all $l$). This gives an upper bound on the first term.

The above argument can be repeated along the unstable direction to control the third term. We get 
\[
    |F(g_{t+s_{p}}u_{p},[qT_l,(q+1)T_l]))-F(g_tx_{p},[qT_l,(q+1)T_l])|
    \leq L_0T_l\kappa\epsilon e^{-\frac{(Q_l-1-q)T_l\eta\alpha}{2}},
\]
and thus
\[
   \begin{aligned}
  |F(g_{t+s_{p}}u_{p},Q_lT_l))-F(g_tx_{p},Q_lT_l)| & \leq \sum_{q=0}^{Q_l-1}L_0T_l\kappa\epsilon e^{-\frac{(Q_l-1-q)T_l\eta\alpha}{2}} \\ &\leq \frac{L_0T_l\kappa \epsilon}{1-e^{-\frac{T_l\eta\alpha}{2}}} \leq KT_l
    \end{aligned}
\]
 To estimate the second term, we observe that

\[
    |F(g_tu_{p},Q_lT_l)-F(g_{t+s_{p}}u_{p},Q_lT_l)|\leq 2||f|| s_{p}\leq 2\kappa \epsilon ||f||.
\]
To estimate the fourth term, we observe that $|F(g_tx_{p},Q_lT_l)-Q_lF(x_{p},T_l)|$ is bounded above by
\[
  \sum_{q=0}^{Q_l-1}\left |\int_{qT_l}^{(q+1)T_l}f(g_{t+s}x_{p})ds-\int_0^{T_l}f(g_tx_{p})dt \right| \leq 2\delta_l Q_l ||f||.
\]
where the last inequality follows because $x_{p}$ has period $t(x_{p}) \in [T_l-\delta_l,T_l]$. By summing the estimates on these four terms, the lemma is proved.
\end{proof}
We also have the following basic comparison between integrals using $\nu_l$ and $m_l$.
\begin{lemma} \label{bestimate2}
We have
\begin{equation}    
    \begin{aligned}
    |\int F_{p}^l d \nu_l-Q_l \int F(\cdot,T_l) d m_l | \leq 2KT_l+2\kappa \epsilon \|f\|+2\delta_l Q_l \|f\|.
    \end{aligned}
\end{equation}
\end{lemma}
\begin{proof}
The expression $|\int F_{p}^l d \nu_l-Q_l \int F(\cdot,T_l) d m_l|$ can be rewritten as 
\[
    \frac{1}{(\# E_l)^{k_l}} \left |\sum_{\ul x \in E_l^{k_l}}\left (\frac{1}{T_l}\int_0^{T_l}F_{p}^l(g_s(\pi_l(\ul x))ds-Q_lF(x_{p},T_l) \right) \right |. 
\]

The results follows using Lemma \ref{bestimate1}. 
\end{proof}

\section{Main Theorem} \label{s.main}

Recall $L_{\nu}(h,c)$ is the Lindeberg function from Definition \ref{Lind}.
\begin{theorem}  \label{main}
Let $(\nu_l)_{l \in \NN}$ be a sequence of measures as constructed in the previous section.
Suppose $f\in C(T^1M)$ is H\"older continuous with
\begin{equation} \label{posvar}
\begin{aligned}
\liminf_{l\rightarrow \infty}\sigma_l^2 > 0.
\end{aligned}
\end{equation}
Then the Lindeberg-type condition
\begin{equation} \label{Lindeberg}
\begin{aligned}
\lim_{l\rightarrow \infty}\frac{\sum_{1\leq p\leq k_l, }L_{\nu_l}(F^l_p,\gamma s_l)}{s_l^2}=0 
\end{aligned}
\end{equation}
 for any $\gamma>0$, implies that for all $a\in \mathbb{R}$, 
\begin{equation} \label{LCLT}
\lim_{l\rightarrow \infty}\nu_l \left ( \left \{v: \frac{F(v, k_l(C_lT_l+M)) - \int F(\cdot,k_l(C_lT_l+M))d \nu_l}{s_l}\leq a \right \} \right )=N(a),
\end{equation}
 where $N$ is the cumulative distribution function of the normal distribution $\mathcal{N}(0,1)$. Conversely, under the hypothesis \eqref{posvar}, \eqref{LCLT} implies \eqref{Lindeberg}.
\end{theorem}

We also show in Lemma \ref{s comparison} that the conclusion \eqref{LCLT} is equivalent to
\begin{equation} 
\lim_{l\rightarrow \infty}\nu_l\left ( \left \{v: \frac{F(v, k_l(C_lT_l+M)) - k_l(C_lT_l+M) \int f d \nu_l}{\sigma_{\nu_l}(F(\cdot, k_l(C_lT_l+M)))}\leq a \right \} \right )=N(a),
\end{equation}
which has the advantage of being in the most elementary possible terms.

The proof of Theorem \ref{main} is based on comparing $\mu_l$ and $\nu_l$. The measures $\mu_l$ are product measures supported on $E_l^{k_l}$ (more precisely, the measures $\mu_l$ are products of the measures $m_l$ on $E_l$) and the version of the results we want can be obtained there by considering these objects as sequences of independent random variables and appealing to classical probability theory. Our main theorem is proved by showing that the relevant quantities for $\nu_l$ are comparable to corresponding quantities for $\mu_l$. The starting point for our main theorem is thus the following theorem on the sequence of measures $(\mu_l)$. We give the Lindeberg condition in terms of the uniform measure $m_l$ on $E_l$ since this is the most elementary object under consideration.

\begin{theorem} \label{mainonmu}
The condition
\begin{equation} \label{mulassump}
\begin{aligned}
\lim_{l\rightarrow \infty}\frac{L_{m_l}(F(\cdot , T_l),\gamma \sqrt{k_l} \sigma_l )}{\sigma_l^2}=0.
\end{aligned}
\end{equation}
holds for any $\gamma >0$ if any only if
\begin{equation} \label{muclt}
\begin{aligned}
\lim_{l\rightarrow \infty}\mu_{l} \left( \left \{ (v_1,\ldots, v_{k_l}) : \frac{\sum_{i=1}^{k_l}F(v_i, T_l)- k_l \int F(\cdot, T_l) d m_l}{\sqrt{k_l \sigma_l^2}}\leq a \right \} \right )=N(a).
\end{aligned}
\end{equation}
\end{theorem}
This result can be obtained formally using an analogous statement of Denker-Senti-Zhang for dynamical arrays. It can be obtained quite easily from the classical Lindeberg CLT \cite{wF68}. We give a short formal proof based on verifying Denker-Senti-Zhang's hypotheses \cite[Proposition 3.3]{DSZ18}.

\begin{proof}
We follow the terminology of  \cite[Proposition 3.3]{DSZ18}. We consider the product of $k_l$ copies of a finite set. We consider a function $G_{l,i}: E_l^{k_l} \to \mathbb R$ which depends on the $i^{th}$ component. That is,
\[
G_{l,i} (x_1, \ldots, x_{k_l}) = G_{l, i} (x_i).
\]
Let $\hat s_l^2 = \sum_{i=1}^{k_l} \sigma^2_{\mu_l}(G_{l,i})$. As stated in \cite[Proposition 3.3]{DSZ18}, it follows from Lindeberg's CLT for independent random variables that the Lindeberg condition holds
\begin{equation} \label{Lind0}
\begin{aligned}
\lim_{l\rightarrow \infty}\frac{\sum_{i=1}^{k_l} L_{\mu_l}(G_{l,i},\gamma \hat s_l)}{\hat s_l^2}=0 \text{\hspace{10pt} for any } \gamma>0
\end{aligned}
\end{equation}
 if and only if  $(G_{l,i})$ is asymptotically negligible
\begin{equation} \label{asympneg}
\begin{aligned}
\lim_{l\rightarrow \infty}\max_{1\leq i \leq k_l}{\frac{ \sigma^2_{\mu_l}(G_{l,i})}{\hat s_l^2}}=0
\end{aligned}
\end{equation}
and for all $a \in \RR$,
\begin{equation} \label{4point8}
\begin{aligned}
\lim_{l\rightarrow \infty}\mu_{l}\left ( \left\{ (x_1, \ldots, x_{k_l}): \frac{\sum_{i=1}^{k_l}G_{l,i}- \int (\sum_{i=1}^{k_l}G_{l,i}) d \mu_l}{\hat s_l}\leq a \right\}\right )=N(a).\end{aligned}
\end{equation}
For our statement, we set $G_{l,i} (x_1, \ldots, x_{k_l}) = F(x_i, T_l)$ for all $1\leq i \leq k_l$. We observe that for each $i$, we have
\[
\sigma^2_{\mu_l} (\ul x \to F(x_i, T_l)) = \sigma^2_{m_l}(F(\cdot, T_l)).
\]
This is because the first expression is
\[
\frac{1}{\# E_l^{k_l}} \sum_{\ul x \in E_l^{k_l}} \left (F(x_i, T_l) - \frac{1}{\# E_l^{k_l}} \sum_{\ul x \in E_l^{k_l}} F(x_i, T_l) \right)^2, 
\]
so using that $\# E_l^{k_l}= (\#E_l)( \#E_l^{k_l-1})$, and that 
\[
\sum_{\ul x \in E_l^{k_l}} F(x_i, T_l) =  \#E_l^{k_l-1} \sum_{x_i \in E_l} F(x_i, T_l),
\]
the result follows.
Thus, $\hat s_l^2 = k_l \sigma^2_l$. The expression $\hat s_l^{-2} \sigma^2_{\mu_l}(G_{l,i})$ in \eqref{asympneg} reduces to $\sigma^2_l/ k_l \sigma^2_l$. Thus asymptotic negligibility is trivially satisfied. The condition \eqref{Lind0} clearly simplifies to \eqref{mulassump}. We obtain the desired statement, noting in \eqref{4point8} that $\int (\sum_{i=1}^{k_l}G_{l,i}) d \mu_l =\sum_{i=1}^{k_l} \int \left (\ul x \to F(x_i, T_l) \right )d \mu_l= k_l \int F( \cdot, T_l) d m_l$.
\end{proof}
The normalization quantities in \eqref{muclt} are stated in terms of $m_l$ to keep them in the most elementary terms. We note that these quantities can be thought of as quantities depending on the product measure $\mu_l$ (rather than $m_l$) since the proof above makes clear that
\[
\int \sum_{i=1}^{k_l}F(v_i, T_l) d \mu_l = k_l \int F( \cdot, T_l) d m_l,
\]
\[
\sum_{i=1}^{k_l} \sigma^2_{\mu_l}(\ul x \to F(x_i, T_l)) = k_l \sigma_l^2.
\]
We now prove a key lemma we will need in order to use Theorem \ref{mainonmu} to describe behavior of the sequence $(\nu_l)$. For $x \in E_l$, define
\[
D_{m_l}(x) := F(x, T_l) - \int F(\cdot,T_l)d m_l.
\]
For $1\leq p \leq k_l$, $t\in [0,T_l]$ and $ \ul x\in E_l^{k_l}$,  define
\[
D_{\nu_l}( \ul x  , t;p) := F_{p}^l(g_t(\pi_l(\ul x)))-\int F_{p}^l d \nu_l, 
\]
and define
\[
\Delta_{p}^l(\ul x, t) :=D_{\nu_l}( \ul x  , t;p) - Q_lD_{m_l}(x_p).
\]
These quantities satisfy the following properties.

\begin{lemma}   \label{bestimate3}
For all $t\in [0,T_l]$ and $ \ul x \in E_l^{k_l}$, we have 
\begin{equation}    \label{deltapl}
    \begin{aligned}
    |\Delta_{p}^l(\ul x, t)|\leq 2(2KT_l+2\kappa \epsilon \| f \|+2\delta_l Q_l \| f \|),
    \end{aligned}
\end{equation}
\begin{equation}    \label{completesquare}
    \begin{aligned}
  D_{\nu_l}( \ul x  , t;p)^2 =   \Delta_{p}^l(\ul x, t)[\Delta_{p}^l(\ul x, t)+2 Q_l D_{m_l}(\ul x, p)] + Q_l^2 D_{m_l}(x_p)^2.
    \end{aligned}
\end{equation}
\end{lemma}
\begin{proof}
The first property follows directly from Lemmas \ref{bestimate1} and \ref{bestimate2}. The second property holds because
\[
\begin{aligned}
  D_{\nu_l}( \ul x  , t;p)^2 &=   D_{\nu_l}( \ul x  , t;p)^2  - Q_l^2 D_{m_l}(x_p)^2 + Q_l^2 D_{m_l}(x_p)^2 \\
  &= \Delta_{p}^l(\ul x, t)[ D_{\nu_l}(\ul x, t;p) + Q_lD_{m_l}(x_p)] + Q_l^2 D_{m_l}(x_p)^2 \\
  &= \Delta_{p}^l(\ul x, t)[\Delta_{p}^l(\ul x, t)+2 Q_l D_{m_l}(\ul x, p)] + Q_l^2 D_{m_l}(x_p)^2. \qedhere
\end{aligned}
\]

\end{proof}

\begin{lemma} \label{keylemma}
$\lim_{l\rightarrow \infty}\frac{\sigma^2_{\nu_l}(F^l_p)}{Q_l^2\sigma_l^2}=1$ uniformly in $1\leq p \leq k_l$.
\end{lemma}
\begin{proof}
Observe that we can write 
\[
\begin{aligned}
Q_l^2\sigma_l^2 & = \frac{1}{\#E_l^{k_l}}\sum_{\ul x \in E_l^{k_l}}\left (Q_lF(x_{p},T_l)-Q_l \int F(\cdot,T_l)d m_l \right )^2 \\
& = \frac{1}{\#E_l^{k_l}}\sum_{\ul x \in E_l^{k_l}}(Q_lD_{m_l}(x_{p}))^2  = \frac{1}{\# E_l^{k_l}}\sum_{\ul x \in E_l^{k_l}}\frac{1}{T_l}\int_0^{T_l}(Q_lD_{m_l}(x_{p}))^2dt.
\end{aligned}
\]
We thus observe using \eqref{deltapl} and \eqref{completesquare} that
\[
\begin{aligned}
\sigma_{\nu_l}^2(F_{p}^l)-Q_l^2\sigma_l^2 
&=\frac{1}{\#E_l^{k_l}}\sum_{\ul x \in E_l^{k_l}}\frac{1}{T_l}\int_0^{T_l}(D_{\nu_l}(\ul x, t;p))^2dt   -Q_l^2\sigma_l^2 \\
&=\frac{1}{\#E_l^{k_l}}\sum_{\ul x \in E_l^{k_l}}\frac{1}{T_l}\int_0^{T_l}((D_{\nu_l}(\ul x, t;p))^2  - Q^2_l D_{m_l}(x_p)^2) dt \\
%&=\frac{1}{\# E_l^{k_l}}\sum_{\ul x \in E_l^{k_l}}\frac{1}{T_l}\int_0^{T_l} \Delta_p( \ul x , t) (D_{\nu_l}(\ul x, t) + Q_lD_{m_l}(x_p) )  dt \\
&=\frac{1}{\# E_l^{k_l}}\sum_{\ul x \in E_l^{k_l}}\frac{1}{T_l}\int_0^{T_l} \Delta_p( \ul x , t) (\Delta_p(\ul x, t) +2 Q_l D_{m_l}(x_p))  dt \\
&=\int \left (\frac{1}{T_l}\int_0^{T_l}\Delta_p^l(\ul x, t)^2dt \right )d \mu_l+\frac{1}{\# E_l^{k_l}}\sum_{\ul x \in E_l^{k_l}}\frac{1}{T_l}\int_0^{T_l}2Q_l \Delta_{p}^l(\ul x, t)D_{m_l}(x_{p})dt \\
&\leq \int \left (\frac{1}{T_l}\int_0^{T_l}\Delta_p^l(\ul x, t)^2dt \right )d \mu_l+2 Q_l \sup_{\ul x, t} \{|\Delta_{p}^l(\ul x, t)|\}  \int D_{m_l} d m_l \\
&\leq \int \left (\frac{1}{T_l}\int_0^{T_l}(\Delta_p^l(\ul x, t))^2dt \right) d \mu_l + 2 Q_l \sigma_l \sup_{\ul x, t}\{|\Delta_{p}^l(\ul x, t)|\} \\
&\leq 4(2KT_l+2\kappa \epsilon \| f \|+2\delta_l Q_l \| f \|)^2 + 4Q_l \sigma_l(2KT_l+2\kappa \epsilon \| f \|+2\delta_l Q_l \| f \|).
\end{aligned}
\]
By Hypothesis \ref{condstar} on $(T_l,k_l,\delta_l,C_l)_{l\in \mathbb{N}}$ and our hypothesis that $\liminf_{l \to \infty} \sigma_l>0$, we conclude that

$$
\lim_{l\rightarrow \infty}\frac{\sigma_{\nu_l}^2(F_{p}^l)-Q_l^2\sigma_l^2}{Q_l^2\sigma_l^2}=0.
$$

Notice that the above upper bound on $\sigma_{\nu_l}^2(F_{p}^l)-Q_l^2\sigma_l^2$ is independent of $p$. As a result, the convergence is uniform in $p$ and this ends the proof of Lemma \ref{keylemma}.
\end{proof}
We obtain the following lemma as an immediate corollary. 
\begin{lemma} \label{keylemmacor} 
The sequence $ s^2_l=\sum_{p} \sigma^2_{\nu_l}(F^l_p), $ satisfies
\begin{equation} \label{eqmain}
    \lim_{l\to \infty} \frac{s_l}{Q_l \sigma_l \sqrt{k_l }}=1.
\end{equation}
\end{lemma}

 We might also consider $s_l'^2:=\sigma_{\nu_l}^2(\sum_{p}F_p^l)$ or $s_l''^2:=\sigma_{\nu_l}^2(F(\cdot, k_l(C_lT_l+M)))$ as natural  substitutes for $s_l^2$. We have the following result
\begin{lemma}   \label{s comparison}
$\lim_{l\rightarrow \infty}\frac{s_l'^2}{s_l^2}=\lim_{l\rightarrow \infty}\frac{s_l''^2}{s_l^2}=1$.
\end{lemma}

\begin{proof} %[Sketch of the proof]
We begin by verifying $\lim_{l\rightarrow \infty}\frac{s_l'^2}{s_l^2}=1$. For any $l>1$ and $1\leq p_1<p_2\leq k_l$, we have
$$
     \int D_{\nu_l}(\ul x,t;p_1)D_{\nu_l}(\ul x,t;p_2) d \nu_l= \int((Q_lD_{m_l}(x_{p_1})+\Delta_{p_1}^l(\ul x,t))(Q_lD_{m_l}(x_{p_2})+\Delta_{p_2}^l(\ul x,t))) d \nu_l.
$$
The right hand side is the sum of four terms, among which $\int Q_l^2D_{m_l}(x_{p_1})D_{m_l}(x_{p_2})d \nu_l=0$, and $\int(Q_lD_{m_l}(x_{p_1}) |\Delta_{p_2}^l(\ul x,t)|d \nu_l \leq 2Q_l\sigma_l(2KT_l+2\kappa \epsilon \| f \|+2\delta_l Q_l \| f \|)$, which also holds true when $p_1$ and $p_2$ are switched, and $\int|\Delta_{p_1}^l(\ul x,t)\Delta_{p_2}^l(\ul x,t)| d \nu_l\leq 4(2KT_l+2\kappa \epsilon \| f \|+2\delta_l Q_l \| f \|)^2$. As a result, we have
$$
\begin{aligned}
&\lim_{l\rightarrow \infty}\left |\frac{s_l'^2}{s_l^2}-1 \right |   =\lim_{l\rightarrow \infty} \left |\frac{2\sum_{1\leq p_1<p_2\leq k_l}\int (D_{\nu_l}(\ul x,t;p_1)D_{\nu_l}(\ul x,t;p_2))}{s_l^2} d \nu_l \right | \\
    &\leq \lim_{l\rightarrow \infty}\frac{8Q_l\sigma_l(2KT_l+2\kappa \epsilon \| f \|+2\delta_l Q_l \| f \|)+4(2KT_l+2\kappa \epsilon \| f \|+2\delta_l Q_l \| f \|)^2}{s_l^2} \\
    &=\lim_{l\rightarrow \infty}\frac{8Q_l\sigma_l(2KT_l+2\kappa \epsilon \| f \|+2\delta_l Q_l \| f \|)+4(2KT_l+2\kappa \epsilon \| f \|+2\delta_l Q_l \| f \|)^2}{Q_l^2\sigma_l^2k_l} \\
    &=0,
    \end{aligned}
$$
where in the second last equality we use Lemma \ref{keylemmacor}. The limit being $0$ follows from Hypothesis \ref{condstar} on $(T_l,k_l,\delta_l,C_l)_{l\in \mathbb{N}}$ and $\liminf_{l \to \infty} \sigma_l>0$.

Now to show $\lim_{l\rightarrow \infty}\frac{s_l''^2}{s_l^2}=1$, it suffices to show that $\lim_{l\rightarrow \infty}\frac{s_l''^2}{s_l'^2}=1$. Write $\Delta'_l(\ul x,t):=\sum_{p=1}^{k_l}\int_{(p-1)(C_lT_l+M)+Q_lT_l}^{p(C_lT_l+M)} f(g_{s+t}\pi_l(\ul x))ds$. Notice that by the definition of $Q_l$, for any $l\geq 1$ and $\ul x\in E_l^{k_l}$ we have
\begin{equation}   \label{Delta'}
    |\Delta'_l(\ul x,t)|\leq k_l((C_l-Q_l)T_l+M)||f||\leq k_l(C_l\delta_l+2T_l+M)||f||.
\end{equation}

We write $D_{\nu_l}(\ul x,t):=\sum_{p=1}^{k_l}(F_p^l(g_t\pi_l(\ul x))- \int F_p^l(g_t\pi_l(\ul x))d \nu_l)=\sum_{p=1}^{k_l} D_{\nu_l}(\ul x,t;p)$. As in the proof of Lemma \ref{keylemma}, we have
$$  
    \begin{aligned}
    &|s_l''^2-s_l'^2| \\
    &=\left |\frac{1}{\# E_l^{k_l}}\sum_{\ul x \in E_l^{k_l}}\frac{1}{T_l}\int_0^{T_l} \Delta_l'( \ul x , t) (\Delta_l'(\ul x, t) +2D_{\nu_l}(\ul x,t)) dt \right | \\
    &=\left |\int \left (\frac{1}{T_l}\int_0^{T_l}(\Delta_l'(\ul x, t))^2dt \right )d \mu_l+\frac{1}{\# E_l^{k_l}}\sum_{\ul x \in E_l^{k_l}}\frac{1}{T_l}\int_0^{T_l}2 \Delta_{l}'(\ul x, t)D_{\nu_l}(\ul x,t)dt \right | \\
    &\leq (k_l(C_l\delta_l+2T_l+M)||f||)^2+2(k_l(C_l\delta_l+2T_l+M)||f||)s_l',
    \end{aligned}
$$
which in turns shows that
$$
\begin{aligned}
&\lim_{l\rightarrow \infty}\left |\frac{s_l''^2}{s_l'^2}-1 \right | \\
&\leq \lim_{l\rightarrow \infty}\frac{(k_l(C_l\delta_l+2T_l+M)||f||)^2+2(k_l(C_l\delta_l+2T_l+M)||f||)s_l'}{s_l'^2}   \\
&=\lim_{l\rightarrow \infty}\frac{(k_l(C_l\delta_l+2T_l+M)||f||)^2+2(k_l(C_l\delta_l+2T_l+M)||f||)Q_l\sigma_l\sqrt{k_l}}{Q_l^2\sigma_l^2k_l}   \\
&=\lim_{l\rightarrow \infty} \frac{k_l^2C_l^2\delta_l^2+2k_l^{3/2}C_lQ_l\sigma_l\delta_l}{Q_l^2\sigma_l^2k_l}=0,
\end{aligned}
$$
and therefore concludes the proof of the lemma.
\end{proof}

Applying Lemma \ref{s comparison}, we can freely replace $s_l$ in \eqref{Lindeberg} and \eqref{LCLT} by $s_l'$ or $s_l''$. Although this is not used in proving Theorem \ref{main}, it allows us to reinterpret the conclusion. We also observe that, in the conclusion, one can easily see that terms of the form $\int F( \cdot, k_l(C_lT_l+M)) d \nu_l$ can be replaced with $k_l(C_lT_l+M) \int f d \nu_l$. 

We now prove the following statement where we compare the average of $F$ along the orbit segment of $v$ over the time interval $[0, k_l(C_lT_l+M)]$ to its average over the sum of the time intervals $[t_p, t_p+Q_lT_l]$.
\begin{lemma} \label{secondlemma}
For each $l\geq 2$, define the functions
\[
A_l (v) := \frac{F(v, k_l(C_lT_l+M)) - \int F( \cdot, k_l(C_lT_l+M)) d \nu_l}{s_l}, 
\]
\[
B_l(v) := \frac{\sum_{p}F_{p}^l(v)- \int \sum_{p}F^l_{p} d \nu_l} {s_l},
\]
where the sum is over $1\leq p\leq k_l$. For any $a>0$, we have 
\[
\lim_{l\rightarrow \infty}\nu_l(v:|A_l-B_l|>a)=0.
\]

\end{lemma}

\begin{proof}
For any constant $a>0$, we have
\[
\begin{aligned}
\lim_{l\rightarrow \infty}\nu_l(v:|A_l(v) -B_l(v) |>a) & \leq \lim_{l\rightarrow \infty}\frac{\int |A_l-B_l| d \nu_l}{a} \\
& \leq \lim_{l\rightarrow \infty}\frac{2k_l(C_l\delta_l+M+2T_l)\| f \|}{a s_l} \\
&=\lim_{l\rightarrow \infty} \frac{2k_l(C_l\delta_l+M+2T_l)\| f \|}{a\sqrt{\sum_{p}\sigma^2_{\nu_l}(F_{p}^{l})}} \\
&=\lim_{l\rightarrow \infty} \frac{2k_l(C_l\delta_l+M+2T_l)\| f \|}{a\sqrt{k_lQ_l^2\sigma_l^2}} \\
&=\lim_{l\rightarrow \infty}\frac{2k_l\delta_l\| f \|}{a\sqrt{k_l}\sigma_l}+\frac{(2M+4T_l)\sqrt{k_l}\| f \|}{aQ_l\sigma_l} =0. 
\end{aligned}
\]
In the above calculation, the second line follows from
\[
\begin{aligned}
\left |F(v, k_l(C_lT_l+M))-\sum_{p}F_{p}^l \right | & \leq k_l(C_lT_l+M-Q_lT_l)\| f \| \\
& \leq k_l(C_lT_l+M-((T_l-\delta_l)C_l T_l^{-1}-2)T_l)\| f \| \\
& = k_l(C_l\delta_l+M+2T_l)\| f \|,
\end{aligned}
\]
the fourth line  follows from Lemma \ref{keylemma}, and  the fifth line converges to $0$ by Hypothesis \ref{condstar} and $\liminf \sigma^2_l>0$.
\end{proof}
Lemma \ref{secondlemma} is the reason we consider sums of the form $\sum_{p=1}^{k_l}F_{p}^l$. We now show that the CLT conclusions for $\mu_l$ and $\nu_l$ are equivalent. For the following proofs, we define a function $Y_p:g_{[0, T_l]}\pi_l(E_l^{k_l}) \to \RR$ by 
\[
Y_p(g_s \pi_l(\ul x)) = F(x_p, T_l) - \int(\ul x \to F(x_i, T_l)) d \mu_l = D_{m_l}(x_p),
\]
and we note that $\sum_{p=1}^{k_l} Y_p(g_s \pi_l(\ul x)) = \sum_{p=1}^{k_l}F(x_p, T_l) - k_l \int F(\cdot, T_l) d m_l$. 

\begin{lemma} \label{conclusionmuiffnu}
The sequence $(\nu_l)$ satisfies the CLT \eqref{LCLT}
\[
\begin{aligned}
\lim_{l\rightarrow \infty}\nu_l(\{x: \frac{F(x, k_l(C_lT_l+M)) - \int F(\cdot,k_l(C_lT_l+M))d \nu_l}{s_l}\leq a\})=N(a).
\end{aligned}
\]
if and only if the sequence $(\mu_l)$ satisfies the CLT \eqref{muclt} 
\[
\begin{aligned}
\lim_{l\rightarrow \infty}\mu_{l} \left( \left \{ (x_1,\ldots, x_{k_l}) : \frac{\sum_{p=1}^{k_l}F(x_p, T_l)- k_l \int F(\cdot, T_l) d m_l}{\sqrt{k_l \sigma_l^2}}\leq a \right \} \right )=N(a).
\end{aligned}
\]
\end{lemma}
\begin{proof}
First we observe that by Lemma \ref{keylemmacor} and Lemma \ref{secondlemma}, and the fact that $\nu_l$ only gives mass to points in $g_{[0, T_l]}\pi_l(E_l^{k_l})$, that the CLT \eqref{LCLT} holds if and only if 
\[
\lim_{l\rightarrow \infty}\nu_l \left ( \left \{g_s(\pi_l(\ul x)) : \ul x \in E_l^{k_l}, s \in [0, T_l], \frac{\sum_{p=1}^{k_l}(F_{p}^l-\int F_{p}^l d \nu_l)}{\sqrt{k_l} Q_l \sigma_l}\leq a \right \} \right )=N(a).
\]
Observe that by \eqref{deltapl} we have
\[
\begin{aligned}
\left | \left (\sum_{p=1}^{k_l}(F_{p}^l-\int F_{p}^l d \nu_l) - Q_l \sum_{p=1}^{k_l} Y_p \right )(g_s(\pi_l(\ul x))) \right | & = \left | \sum_{p=1}^{k_l} \Delta_p(\ul x, t) \right | \\
& \leq 2k_l(2KT_l+2\kappa \epsilon \| f \|+2\delta_lQ_l\| f \|).
\end{aligned}
\]
Fix $b>0$. By Hypothesis \ref{condstar} and \eqref{posvar}, for sufficiently large $l$, 
\begin{equation}    \label{neglectM}
    \frac{2k_l(2KT_l+2\kappa \epsilon \| f \|+2\delta_lQ_l\| f \|)}{\sqrt{k_l}Q_l\sigma_l}< b,
\end{equation}
and it thus follows that for sufficiently large $l$,
\[
\left \{g_s \pi_l(x) :\frac{|\sum_{p=1}^{k_l}(F_{p}^l-\int F_{p}^l d \nu_l)-Q_l\sum_{p=1}^{k_l}Y_p |}{\sqrt{k_l}Q_l\sigma_l}>b \right \}  = \emptyset.
\]
In particular,
\[
\begin{aligned}
&\lim_{l\rightarrow \infty}\nu_l \left ( \left \{g_s \pi_l(x) :\frac{|\sum_{p=1}^{k_l}(F_{p}^l-\int F_{p}^l d \nu_l)-Q_l\sum_{p=1}^{k_l}Y_p |}{\sqrt{k_l}Q_l\sigma_l}>b \right \} \right )  =0.
\end{aligned}
\]
Therefore \eqref{LCLT} holds if and only if 

$$\lim_{l\rightarrow \infty}\nu_l\left ( \left \{g_s \pi_l(\ul x):\ul x \in E_l^{k_l}, s \in [0, T_l], \frac{Q_l\sum_{p=1}^{k_l}Y_p}{\sqrt{k_l}Q_l\sigma_l}\leq a \right \} \right )=N(a).$$

We are now in a position to reformulate in terms of $\mu_l$. Since $Y_p$ does not depend on the variable $s$, then either $g_s \pi_l(\ul x)$ belongs to the above set for all $s \in [0, T_l]$ or for no $s \in [0, T_l]$. It thus follows from the definition of $\nu_l$ that
\[
\nu_l\left ( \left \{g_s \pi_l(\ul x):\frac{Q_l\sum_{p=1}^{k_l}Y_p}{\sqrt{k_l}Q_l\sigma_l}\leq a \right \} \right)= \frac{1}{\#E_l^{k_l}} \# \left \{ \pi_l (\ul x):  \frac{Q_l\sum_{p=1}^{k_l}Y_p}{\sqrt{k_l}Q_l\sigma_l}\leq a \right \}.
\]
Furthermore, by the definition of $Y_p$, we see that
\[
 \left \{ \pi_l (\ul x):  \frac{Q_l\sum_{p=1}^{k_l}Y_p}{\sqrt{k_l}Q_l\sigma_l}\leq a \right \} = \left \{ \ul x \in E_l^{k_l} : \frac{\sum_{p=1}^{k_l}F(x_p, T_l) - k_l \int F(\cdot, T_l) d m_l}{\sqrt{k_l}\sigma_l}\leq a \right \}.
\]
We can thus conclude that
$$\lim_{l\rightarrow \infty}\mu_l \left ( \left \{\ul x:\frac{\sum_{p=1}^{k_l}F(x_p, T_l) - k_l \int F(\cdot, T_l) d m_l}{\sqrt{k_l}\sigma_l}\leq a \right \} \right )=N(a).$$
Thus, we conclude that \eqref{LCLT} holds if and only if \eqref{muclt} holds. 
\end{proof}
All that remains to show equivalence of the Lindeberg conditions in Theorem  \ref{main} on $(\nu_l)$ and in Theorem \ref{mainonmu} on $(\mu_l)$.

\begin{lemma} \label{hypmuiffnu}
If $\liminf_{l \to \infty}\sigma_l>0$, then the Lindeberg condition \eqref{Lindeberg}
\[
\lim_{l\rightarrow \infty}\frac{\sum_{1\leq p\leq k_l, }L_{\nu_l}(F^l_p,\gamma s_l)}{s_l^2}=0 
\]
holds for all $\gamma>0$ if and only if the Lindeberg condition \eqref{mulassump}
\[
\lim_{l\rightarrow \infty}\frac{L_{m_l}(F(\cdot , T_l),\gamma \sqrt{k_l} \sigma_l )}{\sigma_l^2}=0
\]
holds for all $\gamma>0$.
\end{lemma}

\begin{proof}

Let $Z_l(c)=Z(c, F^l_p, \nu_l) = \{x : |F^l_p - \int F^l_p d \nu_l| > c\}$ be the set from the Lindeberg condition. Observe that
\[
    \begin{aligned}
       L_{\nu_l}(F^l_{p},\gamma s_l) &= \int (F^l_p - \int F^l_p d \nu_l)^2 \mathbb{1}_{Z_l(\gamma s_l)} d \nu_l \\
        & = \frac{1}{E_l^{k_l}} \sum_{\ul x \in E_l^{k_l}} \frac{1}{T_l} \int_0^{T_l} D_{\nu_l}(\ul x, t;p)^2 \mathbb{1}_{Z_l(\gamma s_l)}(g_t\pi_l(\ul x)) dt.
\end{aligned}
\]
Using \eqref{completesquare}, we see that $L_{\nu_l}(F^l_{p},\gamma s_l)$ is bounded above by the sum of the terms
\[
\frac{1}{E_l^{k_l}} \sum_{\ul x \in E_l^{k_l}} \frac{1}{T_l}  \int_0^{T_l}  \Delta_{p}^l(\ul x, t)[\Delta_{p}^l(\ul x, t)+2 Q_l D_{m_l}(\ul x, p)] dt,
\] 
and 
\[
\frac{1}{E_l^{k_l}} \sum_{\ul x \in E_l^{k_l}}   \frac{1}{T_l}  \int_0^{T_l}Q_l^2 D_{m_l}(x_p)^2\mathbb{1}_{Z(\gamma s_l)}(g_t\pi_l(\ul x)) dt.
\]
The first of these terms is equal to $\sigma^2_{\nu_l}(F^l_p) - Q_l^2 \sigma_l$ as observed in the proof of Lemma \ref{keylemma}. The second term can be written as 
\[
\int(Q_lY_p)^2 \mathbb{1}_{Z_l(\gamma s_l)} d \nu_l.
\] 
 Since $s_l^{-2} \Sigma_p(\sigma^2_{\nu_l}(F^l_p) - Q_l^2 \sigma_l^2) \to 0$ by the proof of Lemma \ref{keylemma},  it follows that 
\[
    \begin{aligned}
    \lim_{l\rightarrow \infty}\frac{\sum_{p}L_{\nu_l}(F^l_{p},\gamma s_l)}{s_l^2} \leq \lim_{l\rightarrow \infty}\frac{\sum_{p} \int(Q_l Y_p)^2\mathbb{1}_{Z_l(\gamma s_l)}d \nu_l}{s_l^2}.
    \end{aligned}
\]

We now work on the set $Z_l(\gamma s_l)$. Since $\nu_l (\{ g_t (\pi_l (\ul x)) : \ul x \in E^{k_l}_l, t \in [0, T_l]\}) =1$, it suffices for our argument to consider the set
\[
Z_l'(\gamma s_l):= \{ g_t (\pi_l (\ul x)) : \ul x \in E^{k_l}_l, t \in [0, T_l], |F^l_p - \int F^l_p d \nu_l| > \gamma s_l\}.
\]
Note that $ |F^l_p  g_t (\pi_l (\ul x)) - \int F^l_p d \nu_l| = |D_{\nu_l}(\ul x, t)| \leq |\Delta^l_p(\ul x, t)| + Q_l|Y_p(  g_t (\pi_l (\ul x))|$. Thus
\[
Z_l'(\gamma s_l) \subset \{g_t (\pi_l (\ul x)) : |Y_p(g_t (\pi_l (\ul x))| \geq Q_l^{-1}(\gamma s_l - |\Delta^l_p(\ul x, t)|) \}.
\]
Recall that $\sup_{\ul x, t}\{|\Delta_{p}^l (\ul x, t)|\}\leq 2(2KT_l+\kappa \epsilon \| f \|+2\delta_l Q_l \| f \|)$ and $\lim_{l\rightarrow \infty} \frac{s_l}{\sqrt{k_l}Q_l\sigma_l}=1$.  Therefore, by Hypothesis \ref{condstar} and \eqref{posvar}, for sufficiently large $l$, we have $|\Delta_p^l(\ul x, t)| \leq \frac{\gamma s _l}{2}$ for all $t\in [0,T_l]$ and $\ul x \in E_l^{k_l}$. It follows that for sufficiently large $l$,

\begin{equation}    \label{setzl}
\begin{aligned}
Z_l'(\gamma s_l) & \subset \{g_t (\pi_l (\ul x)) : |Y_p(g_t (\pi_l (\ul x))| \geq  \gamma s_l(2Q_l)^{-1} \} \\
& \subset  \{g_t (\pi_l (\ul x)) : |Y_p(g_t (\pi_l (\ul x))| \geq  (\gamma \sigma_l \sqrt{k_l})/4 \}.
\end{aligned}
\end{equation}
Thus for all large $l$,
\begin{equation}  \label{eqQlYpupperbound}
    \begin{aligned}
\int (Q_l Y_p)^2 \mathbb{1}_{Z_l(\gamma s_l)} d \nu_l & = \int (Q_l Y_p)^2 \mathbb{1}_{Z'_l(\gamma s_l)} d \nu_l \\
& \leq \int (Q_l Y_p)^2 \mathbb{1}_{ \{g_t (\pi_l (\ul x)) : |Y_p(g_t (\pi_l (\ul x))| \geq  (\gamma \sigma_l \sqrt{k_l})/4 \}} d \nu_l \\
& = Q_l^2 \int D_{m_l} ((\ul x \to x_p))^2 \mathbb{1}_{\{ \ul x: |D_{m_l}(x_p)| \geq  (\gamma \sigma_l \sqrt{k_l})/4 \}} d \mu_l, \\
& = Q_l^2 \int D_{m_l} (x)^2 \mathbb{1}_{\{ x: |D_{m_l}(x)| \geq  (\gamma \sigma_l \sqrt{k_l})/4 \}} d m_l \\
& = Q_l^2 L_{m_l}(F(\cdot , T_l),\gamma \sigma_l  \sqrt{k_l}/4).
\end{aligned}
\end{equation}
Combining the above calculations, and using \eqref{eqmain}, it follows that if we assume  \eqref{mulassump}, then
\[
    \begin{aligned}
    \lim_{l\rightarrow \infty}\frac{\sum_{p}L_{\nu_l}(F^l_{p},\gamma s_l)}{s_l^2} &\leq \lim_{l\rightarrow \infty}\frac{\sum_{p} \int(Q_l Y_p)^2\mathbb{1}_{Z_l(\gamma s_l)}d \nu_l}{s_l^2} \\
    & \leq \lim_{l\rightarrow \infty} \frac{k_l Q_l^2 L_{m_l}(F(\cdot , T_l),\gamma \sigma_l  \sqrt{k_l}/4)}{s_l^2}\\
     & = \lim_{l\rightarrow \infty} \frac{L_{m_l}(F(\cdot , T_l),\gamma \sigma_l  \sqrt{k_l}/4)}{\sigma_l^2} =0, 
    \end{aligned}
\]
and thus \eqref{Lindeberg} is true. 

To check \eqref{Lindeberg} $\implies$ \eqref{mulassump}, note that $L_{\nu_l}(F^l_{p},\gamma s_l)$ is bounded below by the sum of
\[
-\frac{1}{E_l^{k_l}} \sum_{\ul x \in E_l^{k_l}} \frac{1}{T_l}  \int_0^{T_l}  \Delta_{p}^l(\ul x, t)[\Delta_{p}^l(\ul x, t)+2 Q_l D_{m_l}(\ul x, p)] dt,
\] 
and 
\[
\frac{1}{E_l^{k_l}} \sum_{\ul x \in E_l^{k_l}}   \frac{1}{T_l}  \int_0^{T_l}Q_l^2 D_{m_l}(x_p)^2\mathbb{1}_{Z(\gamma s_l)}(g_t\pi_l(\ul x)) dt.
\]
As in the discussion above, we have
$$
    \begin{aligned}
    \lim_{l\rightarrow \infty}\frac{\sum_{p=1}^{k_l}L_{\nu_l}(F^l_{p},\gamma s_l)}{s_l^2} \geq \lim_{l\rightarrow \infty}\frac{\sum_{p=1}^{k_l} \int(Q_l Y_p)^2\mathbb{1}_{Z_l(\gamma s_l)}d \nu_l}{s_l^2}.
    \end{aligned}
$$
We also have $ |F^l_p  g_t (\pi_l (\ul x)) - \int F^l_p d \nu_l| \geq -|\Delta^l_p(\ul x, t)| + Q_l|Y_p(  g_t (\pi_l (\ul x))|$, which implies that
$
\{g_t (\pi_l (\ul x)) : |Y_p(g_t (\pi_l (\ul x))| \geq Q_l^{-1}(\gamma s_l + |\Delta^l_p(\ul x, t)|) \} \subset Z_l'(\gamma s_l).
$

Since $|\Delta^l_p(\ul x, t)|\leq \gamma s_l$ for all $\ul x\in E_l^{k_l}$ and $t\in [0,T_l]$ when $l$ is sufficiently large, we have
$
\{g_t (\pi_l (\ul x)) : |Y_p(g_t (\pi_l (\ul x))| \geq 2Q_l^{-1}\gamma s_l \} \subset Z_l'(\gamma s_l).
$
By \eqref{eqmain}, we have
$$
\{g_t (\pi_l (\ul x)) : |Y_p(g_t (\pi_l (\ul x))| \geq 4\gamma\sqrt{k_l}\sigma_l \}\subset Z_l'(\gamma s_l).
$$
Then following the same argument as in \eqref{eqQlYpupperbound}, we have
$$
\int (Q_l Y_p)^2 \mathbb{1}_{Z_l'(\gamma s_l)} d \nu_l\geq Q_l^2 L_{m_l}(F(\cdot , T_l),4\gamma \sigma_l  \sqrt{k_l}),
$$
which shows that
$$
    \begin{aligned}
    \lim_{l\rightarrow \infty}\frac{\sum_{p=1}^{k_l}L_{\nu_l}(F^l_{p},\gamma s_l)}{s_l^2} \geq \lim_{l\rightarrow \infty} \frac{L_{m_l}(F(\cdot , T_l),4\gamma \sigma_l  \sqrt{k_l})}{\sigma_l^2}.
    \end{aligned}
$$
This shows that \eqref{Lindeberg} implies \eqref{mulassump}. %As a result, the hypotheses \eqref{mulassump} and \eqref{Lindeberg} are equivalent. 
\end{proof}

\section{Verifying the Lindeberg condition} \label{Verifylind}

Historically, the Lindeberg  CLT is used in the case where an underlying probabilistic mixing structure is available (see condition (I) and (II) in \cite{ib62} for definitions of mixing and K-property in probability). In those situations, given any $L^1$ random variable $f$, to evaluate the distribution of a sum $S_nf$, one observes its partial sums $(S_{a_i}^{b_i}f)_{i\in \mathbb{N}}$, where $0=a_0<b_0<a_1<\cdots$. Due to the mixing assumptions on the system, one can expect $S_{a_i}^{b_i}f$ to behave `independently' for different $i\in \mathbb{N}$, if $a_{i+1}-b_i$, which is the gap between $i$-th and $i+1$-th segment, increases to $\infty$ uniformly for all $i\in \mathbb{N}$. To make $S_nf$ well-approximated by the sum over $S_{a_i}^{b_i}f$, it is natural to consider $b_i-a_i\gg a_{i+1}-b_i$ for all $i\in \mathbb{N}$ so that the effect from the gap is negligible. See Theorem 1.3 in \cite{ib62}. In particular, for $f$ with finite $2+\delta$ moments and $\sigma^2(S_nf)$ tending to infinity, the Lindeberg condition is satisfied. The mixing structure of the system allows one to argue that the Lindeberg variance distributed by each segment individually is sub-linear compared to the total variance, while mixing also implies the growth of total variance is (almost) linear. Therefore, the overall Lindeberg variance is negligible.   

In our situation, we do not have any strong mixing properties available for the measures $(\nu_l)$. However, each $\nu_l$ is weighted over concatenations of $k_l$ segments of (repeated) independent closed geodesics with (approximately) $T_l$ length, so one can study the global Lindeberg condition \eqref{Lindeberg} via the local condition \eqref{mulassump}. Intuitively, if we can make $k_l$ increase at an appropriate rate compared to $T_l$, eventually the Lindeberg variance contributed by individual terms becomes negligible, and thus the local condition \eqref{mulassump} is satisfied.

From now on, we strengthen condition \eqref{posvar} to the following 
\begin{equation}    \label{infvar}
    \lim_{l\rightarrow \infty}\sigma_l^2=\infty.
\end{equation} 
With this assumption, we can weaken the condition $k_l\delta_l^2 \downarrow 0$ in Hypothesis \ref{condstar} to
\begin{equation}    \label{newhypo}
    \frac{k_l\delta_l^2}{\sigma_l^2} \downarrow 0,
\end{equation}
and still obtain Theorem \ref{main}. This is because wherever the old condition $k_l\delta_l^2 \downarrow 0$ is applied, we are actually dealing with the limit of $k_l\delta_l^2/\sigma_l^2$ (see the last line in the proof of Lemma \ref{s comparison}, the fifth line of Lemma \ref{secondlemma}, \eqref{neglectM} in Lemma \ref{conclusionmuiffnu} and \eqref{setzl} in Lemma \ref{hypmuiffnu}). With the new assumption \eqref{newhypo}, we can allow $k_l$ to grow faster than before. If we can find $k_l$ which satisfies \eqref{newhypo} while simultaneously satisfying the hypothesis of the following lemma, we are done.
\begin{lemma} \label{hypoforlind}
Suppose that we have chosen $\delta_l$ and $T_l$, and our observable $f$, and that $\sigma_l \to \infty$. Suppose we can find $k_l \to \infty$ so that $\frac{ \sqrt k_l \sigma_l}{T_l} \to \infty$. Then the Lindeberg condition \eqref{mulassump} is satisfied.
\end{lemma}
\begin{proof}
We consider the Lindeberg condition \eqref{mulassump}. For any fixed $\gamma>0$ and $v\in T^1M$, the indicator function in the integral satisfies
\begin{equation}   \label{indicator}
    \begin{aligned}
    \mathbb{1}_{|F(\cdot,T_l)-\int F(\cdot,T_l)dm_l|\geq \gamma\sqrt{k_l}\sigma_l}(v)
    \leq \mathbb{1}_{2T_l||f||\geq \gamma\sqrt{k_l}\sigma_l}(v)
    = \mathbb{1}_{K_{\gamma,f}\geq T_l^{-1}\sqrt k_l \sigma_l }(v)
    \end{aligned}
\end{equation}
where $K_{\gamma,f}:=2||f||\gamma^{-1}$ is a constant. Thus,
\begin{equation}    \label{verilind}
\begin{aligned}
&\lim_{l\rightarrow \infty}\frac{L_{m_l}(F(\cdot , T_l),\gamma \sqrt{k_l} \sigma_l )}{\sigma_l^2} \\
&=\lim_{l\rightarrow \infty}\frac{\int (F(\cdot,T_l)-\int F(\cdot,T_l)dm_l)^2\mathbb{1}_{|F(\cdot,T_l)-\int F(\cdot,T_l)dm_l|\geq \gamma\sqrt{k_l}\sigma_l} dm_l}{\sigma_l^2} \\
&\leq \lim_{l\rightarrow \infty}\frac{\int (F(\cdot,T_l)-\int F(\cdot,T_l)dm_l)^2\mathbb{1}_{K_{\gamma,f}\geq T_l^{-1}\sqrt k_l \sigma_l } dm_l}{\sigma_l^2} \\
&=0 \\
\end{aligned}
\end{equation}
which verifies Lindeberg condition \eqref{mulassump}
\end{proof}
Recall that we defined the (lower) dynamical variance for the sequence of measures $(m_l)$ to be
\begin{equation} \label{dynvarml}
 \underline \sigma^2_{\text{Dyn}}(f, (m_l)) = \liminf_{l \to \infty} \int \left (\frac{F(\cdot, T_l) - \int F( \cdot, T_l )d m_l}{\sqrt{T_l}} \right)^2 d m_l =  \liminf_{l \to \infty}\frac{\sigma_l^2}{T_l}
\end{equation}
See the introduction for a discussion of this quantity.
%In hyperbolic settings, the dynamical variance \eqref{dynvar} is positive when the observable is not cohomologous to a constant, which is a generic condition. This gives us intuition that \eqref{dynvarml} should reasonably be expected to be positive for generic observables, although such a theory, and its connection with cohomology, is out of reach of current techniques.

\begin{theorem} \label{automaticlindeberg}
Suppose that we have chosen $\delta_l$ and $T_l$, and our observable $f$. Suppose that $\underline \sigma^2_{\text{Dyn}}(f, (m_l))>0$. Then there exists sequences $k_l \to \infty$ and $C_l \to \infty$ so that the measures $(\nu_l)$ constructed from the data $(\delta_l, T_l, k_l, C_l)_{l \in \NN}$ are valid for Theorem \ref{main} to hold, and so that the Lindeberg condition \eqref{Lindeberg} holds.
\end{theorem}
\begin{proof}
We let $k_l := \sigma_l^2/\delta_l$, which clearly tends to $\infty$. Observe that  $\frac{k_l\delta_l^2}{\sigma_l^2} = \delta_l \downarrow 0$, and thus \eqref{newhypo} is satisfied. Making any suitable choice of $C_l$, it follows from the discussion above that Theorem \ref{main} is valid for the measures $(\nu_l)$ constructed from the data $(\delta_l, T_l, k_l, C_l)_{l \in \NN}$.

Observe that from the hypothesis that $\underline \sigma^2_{\text{Dyn}}(f, (m_l))>0$, the sequence $\frac{\sigma_l}{T_l}$ is eventually greater than some constant $\alpha >0$, and thus we have 
\[
\frac{ \sqrt k_l \sigma_l}{T_l} \to \infty.
\]
Thus the hypothesis of Lemma \ref{hypoforlind} is satisfied, and we can conclude that the Lindeberg condition \eqref{mulassump} on $(m_l)$ holds. Thus, by Lemma \ref{hypmuiffnu}, the Lindeberg condition \eqref{Lindeberg} holds on $(\nu_l)$.
\end{proof}
Combining Theorem \ref{main} and Theorem \ref{automaticlindeberg} gives us Theorem \ref{intromain} as an immediate consequence.

\begin{remark}
One can investigate when the Lindeberg condition holds under the weaker condition that  $\lim_{l\rightarrow \infty}\sigma_l^2=\infty$ without assuming that $\underline \sigma^2_{\text{Dyn}}(f, (m_l))>0$. It can be verified that a suitable sequence $(k_l)$ satisfying Lemma \ref{hypoforlind} can be found if $\sigma_l^2/\delta_lT_l\to \infty$. To verify this condition, first recall from Hypothesis \ref{condstar} that the choice on $T_l$ is only determined by $\delta_l$. Thus, we need information on how $T_{\delta_l}$ is related to $\delta_l$ as $\delta_l \to 0$. This information can be extracted in the uniform case using symbolic dynamics, and the issue does not appear at all in discrete-time analogues of this result. While it may be possible to use this criterion to slightly sharpen our results in some concrete examples where the relationship between $\delta_l$ and $T_{\delta_l}$ is tractable, we do not pursue this at present.
\end{remark}

\section{Extensions of main result} \label{ES}
In this section, we extend our main result to dynamical arrays of observables. We also discuss how our techniques extend to equilibrium states and how they apply to other classes of dynamical system beyond geodesic flow.

\subsection{Dynamical Arrays}

A benefit of the Lindeberg approach is that we can consider dynamical arrays in the CLT instead of a single function. In this section, our setup is as follows. We let $(f_l)_{l\in \mathbb{N}}$ be a sequence of H\"older continuous observables. We allow for different H\"older constants and exponents, not necessarily bounded away from $\infty$ and $0$ respectively. We let $L_l$ and $\alpha_l$ be the H\"older constant and exponent respectively for $f_l$, so that $|f_l(x)-f_l(y)|\leq L_ld(x,y)^{\alpha_l}$ for all $l\in \mathbb{N}$.

Given  a sequence of 4-tuples $(T_l,k_l,\delta_l,C_l)_{l\in \mathbb{N}}$ to be chosen precisely later, and the sequence of observables $(f_l)$, we write $F_l(v,T_l):=\int_0^{T_l}f_l(g_t(v))dt$, and $F_{p,q}^l(v):=\int_{t_p+qT_l}^{t_p+(q+1)T_l}f_l(g_t(v))dt$. Using these modified definitions, new definitions for $\sigma_l^2$, $F_p^l$ and $s_l^2$ follow as in $\mathsection 3.1$. We have the following analogy to the statement of Lemma \ref{bestimate1}, with only minor modifications to the proof.
\begin{lemma}      \label{array bestimate1}
For $(f_l)_{l\in \mathbb{N}}$ given as above and $\ul x \in E_l^{k_l}$, $1\leq p \leq k_l$, we have
$$
|F^l_p (g_t(\pi_l(\ul x))) - Q_l F_l( x_{p}, T_l) | \leq 2 K_lT_l + (\kappa \eps +  2 \delta_l Q_l) \| f_l \|,
$$
where $K_l:=L_l\kappa\epsilon(1-e^{-\frac{\eta\alpha_l}{2}})^{-1}$.
\end{lemma}

We need to modify our assumptions on the sequence of 4-tuples $(T_l,k_l,\delta_l,C_l)_{l\in \mathbb{N}}$. 

\begin{hypothesis} \label{arraytuple}
We choose sequences $T_l \in (0, \infty)$, $k_l \in \mathbb{N}$, $\delta_l \in (0, \delta_0)$, and $C_l \in \NN$  which satisfy the following relationships:

1)   For all $l\in \mathbb{N}$, $T_l>\max\{T_0(\delta_l,\eta),1\}$,

2) $T_l \uparrow \infty$, $\frac{T_l}{T_0(\delta_l,\eta)}\uparrow \infty$ and $k_l \uparrow \infty$,

3)  $k_l\delta_l^2\max\{ \| f_l \|,1\} \downarrow 0$,

4) $\frac{\sqrt{k_l}T_l\max\{|K_l|,1\}}{Q_l}\downarrow 0$ and  $\frac{\sqrt{k_l}T_l\max\{ \| f_l \|,1\}}{Q_l}\downarrow 0$.
\end{hypothesis}

It is always possible to have such sequence of 4-tuples as we can first choose $k_l$, then $\delta_l$ and $T_l$, finally $Q_l$. We will demonstrate why we choose $(T_l,k_l,\delta_l,C_l)$ this way below. We have the following analogy to Theorem \ref{main}:

\begin{theorem} \label{arraymain}
Fix $(f_l)_{l\in \mathbb{N}}$ as above. Let $(T_l,k_l,\delta_l,C_l)_{l\in \mathbb{N}}$ be a sequence satisfying Hypothesis \ref{arraytuple} and $(\nu_l)_{l\in \mathbb{N}}$ be the sequence of measures constructed as in $\mathsection 3$. Suppose $(f_l)_{l\in \mathbb{N}}$ satisfies
\begin{equation} \label{arrayposvar}
\begin{aligned}
\liminf_{l\rightarrow \infty}\sigma_l^2 > 0.
\end{aligned}
\end{equation}
Then the Lindeberg-type condition
\begin{equation} \label{arrayLindeberg}
\begin{aligned}
\lim_{l\rightarrow \infty}\frac{\sum_{1\leq p\leq k_l, }L_{\nu_l}(F^l_p,\gamma s_l)}{s_l^2}=0 
\end{aligned}
\end{equation}
 for any $\gamma>0$, implies that for all $a\in \mathbb{R}$, 
\begin{equation} \label{arrayLCLT}
\lim_{l\rightarrow \infty}\nu_l(\{v: \frac{F_l(v, k_l(C_lT_l+M)) - \int F_l(\cdot,k_l(C_lT_l+M))d \nu_l}{s_l}\leq a\})=N(a),
\end{equation}
 where $N$ is the cumulative distribution function of the normal distribution $\mathcal{N}(0,1)$. Conversely, under the hypotheses \eqref{arrayposvar}, \eqref{arrayLCLT} implies \eqref{arrayLindeberg}.
\end{theorem}

%Notice that all the notations above are adapted to the case of dynamical arrays. 
The proof follows the arguments of $\mathsection 4$, with $F$ replaced by $F_l$ and other notations referring to the array version of the definitions. We point out where the differences appear in the proofs between Theorem \ref{arraymain} and Theorem \ref{main}.

We inherit the definitions of $D_{m_l}(x)$, $D_{\nu_l}(\ul x,t;p)$ and $\Delta_{p}^l(\ul x,t)$ from $\mathsection 4$, which all adapt to the dynamical array setting. Observe that as a direct consequence of Lemma \ref{array bestimate1}, \eqref{deltapl} in Lemma \ref{bestimate3} now becomes
\begin{equation}    \label{array bestimate2}
    |\Delta_{p}^l(\ul x, t)|\leq 2(2K_lT_l+2\kappa \epsilon \| f_l \|+2\delta_l Q_l \| f_l \|).
\end{equation}

Therefore, to conclude the main lemma, which says that $\lim_{l\rightarrow \infty}\frac{\sigma^2_{\nu_l}(F^l_p)}{Q_l^2\sigma_l^2}=1$ uniformly in $1\leq p \leq k_l$, it suffices to show $\lim_{l\rightarrow \infty}\frac{2(2K_lT_l+\kappa \epsilon \| f_l \|+2\delta_l Q_l \| f_l \|)}{Q_l\sigma_l}=0$. This can be observed from the proof of Lemma \ref{keylemma}, using Hypothesis \ref{arraytuple} and \eqref{arrayposvar}. As a simple follow-up we have
\begin{equation} \label{arraymainlemma}
    \lim_{l\rightarrow \infty}\frac{s_l^2}{Q_l^2\sigma_l^2k_l}=1.
\end{equation}

To retrieve the content of Lemma \ref{secondlemma}, it suffices to show the last step of its proof holds true in the array case, which is that
\[
\lim_{l\rightarrow \infty}\left (\frac{2k_l\delta_l\| f_l \|}{\sqrt{k_l}\sigma_l}+\frac{(2M+4T_l)\sqrt{k_l}\| f_l \|}{Q_l\sigma_l} \right ) =0. 
\] 
This is obtained by applying condition 3) in Hypothesis \ref{arraytuple} to the first half, condition 4) to the second and applying \eqref{arrayposvar}.

To verify the equivalence between the CLT for $(\nu_l)$ and $(\mu_l)$, which is Lemma \ref{conclusionmuiffnu}, it suffices to replace \eqref{neglectM} by showing $ \frac{2k_l(2K_lT_l+\kappa \epsilon \| f_l \|+2\delta_lQ_l\| f_l \|)}{\sqrt{k_l}Q_l\sigma_l}< b$ for any $b>0$ when $l$ is sufficiently large. Finally, to verify the equivalence of the Lindeberg conditions, analogous to Lemma \ref{hypmuiffnu}, we invoke  \eqref{array bestimate2} and \eqref{arraymainlemma} along with Hypothesis \ref{arraytuple} and \eqref{arrayposvar}. As a result, we are able to conclude that Theorem \ref{arraymain} holds.

\subsection{Equilibrium States}  We refer the reader to \cite{BCFT} for definitions and notations. We consider a potential function $\varphi$ that is either H\"older continuous  or $q\varphi^u$ with $q<1$, where $\varphi^u$ is the geometric potential. We assume that the pressure gap condition $P(\text{Sing},\varphi)<P(\varphi)$ holds. Theorem A in \cite{BCFT} shows that the geodesic flow has a unique equilibrium state $\mu_{\varphi}$. Our main result, Theorem $\ref{main}$, extends to equilibrium states of this type. The generalization is a natural one. In place of the measures $(m_l)$, we use weighted measures  
\[
\hat m_l := \frac{1}{\sum_{v \in E_l} e^{\Phi(v, T_l)} } \sum_{v \in E_l} e^{\Phi(v, T_l)} \delta_v,
\]
and we define a weighted sequence of measures $(\hat \nu_l)$ analogously to our definition of $(\nu_l)$. We can show that $(\hat \nu_l)$ converges to $\mu_{\varphi}$, and that we have the analogue of Theorem $\ref{main}$: if the variance of an observable $f$ with respect to the sequence $(\hat m_l)$ is positive, we can ensure that the sequence $(\hat \nu_l)$ satisfies \eqref{LCLT}. The details of the statement and proof can be found in the PhD thesis of T. Wang \cite{tWthesis}.

\subsection{Systems with non-uniform specification} The reader will have observed that our arguments used dynamical structure proved in \cite{BCFT} rather than direct geometric arguments, and thus it is clear that the arguments of this paper will apply to a variety of systems other than the geodesic flow on non-positive curvature manifolds. We do not attempt to make an general statement abstracting the properties of the geodesic flow used in our analysis - a main point of course is the non-uniform specification structure obtained in \cite{BCFT}. The interested reader can infer from \S \ref{background}-\S \ref{s.main} exactly what properties are needed to obtain this Lindeberg-type CLT on periodic orbits for other systems. In \cite{CT19}, we defined $\lambda$-decompositions as an abstraction of the non-uniform structure enjoyed by rank one geodesic flows. Systems admitting this kind of structure are prime candidates for this kind of analysis. We note that our arguments are all given for flows, but could also be given in the simpler discrete-time case. In discrete-time, one advantage of our construction is that it extends easily from the MME case to equilibrium states.% - this was less clear in the construction of \cite{DSZ18}. 

\bibliographystyle{siam}
\bibliography{LindCLT}

\end{document}